\documentclass[]{interact}

\usepackage{epstopdf}

\usepackage[numbers,sort&compress]{natbib}

\usepackage{setspace}

\usepackage{setspace}
\usepackage{graphicx}

\usepackage{mathrsfs}

 \usepackage{color,xcolor}
\usepackage{tikz}
\usepackage{float}
\usepackage{amscd}
\usepackage{amsfonts}
\usepackage{latexsym}
\usepackage{amssymb}
\usepackage{epsf}
\usepackage{indentfirst}
\usepackage{mathrsfs}
\usepackage{graphicx}
\usepackage{amsmath}
\usepackage{epstopdf} 
\usepackage{mathptmx}      

\usepackage{chngcntr}

\usepackage{caption}

\DeclareCaptionLabelFormat{nopunct}{#1~#2}       
\bibpunct[, ]{[}{]}{,}{n}{,}{,}
\renewcommand\bibfont{\fontsize{10}{12}\selectfont}

\theoremstyle{plain}
\newtheorem{theorem}{Theorem}[section]

\theoremstyle{definition}

\newtheorem{example}[theorem]{Example}

\theoremstyle{remark}
\newtheorem{remark}{Remark}

\begin{document}


\title{A class of flexible and efficient partitioned Runge-Kutta-Chebyshev methods for some time-dependent partial differential equations}

\author{
	\name{Xiao Tang\textsuperscript{a}\thanks{Corresponding author: Xiao Tang. Email: tangx@xtu.edu.cn, 1293060753@qq.com} and Junwei Huang\textsuperscript{a}}
	\affil{\textsuperscript{a}School of Mathematics and Computational Science \& Hunan Key Laboratory for Computation and Simulation in Science and Engineering, Xiangtan University, Hunan 411105, China. }
}

\maketitle

\begin{abstract}
Many time-dependent partial differential equations (PDEs) can be transformed into an ordinary differential equations (ODEs) containing moderately stiff and non-stiff terms after spatial semi-discretization.
In the present paper, we construct a new class of second-order partitioned explicit stabilized methods for the above ODEs.
We treat the moderately stiff term with an $s$-stage Runge-Kutta-Chebyshev (RKC) method and treat the non-stiff term with a $4m$-stage explicit Runge-Kutta (RK) method.
Different from several existing partitioned explicit stabilized methods that employ fixed-stage RK methods to handle the non-stiff term,
both the parameters $s$ and $m$ in our methods can be flexibly adjusted as needed for the problems.
This feature endows our methods with superior flexibility and applicability compared to several existing partitioned explicit stabilized methods, as demonstrated
in several specific numerical examples (including the advection-diffusion equations, the Burgers equations, the Brusselator equations and the damped wave equations).
\end{abstract}


\section{Introduction}
\label{sec1}

In this paper, we focus on the partitioned explicit stabilized methods for the ordinary differential equation (ODE) with the form
\begin{flalign}
	y'(t) = f_D(y(t))+f_A(y(t)), 	\quad y(t_0) = a\in \mathbb{R}^d, \quad t \in [t_0, T], \label{eq:1.1}
\end{flalign}
where $f_D$ and $f_A$ represent the moderately stiff and non-stiff terms, respectively.
An autonomous form is considered here for simplicity. We emphasize that the methods proposed in this paper can be readily extended to the non-autonomous form.
Many time-dependent partial differential equations (PDEs) can be transformed into an ODE systems of the form \eqref{eq:1.1} after spatial semi-discretization,
such as the advection-diffusion equations, the Burgers equations, the Brusselator equations and the damped wave equations.

Due to the presence of the stiff term, the traditional explicit methods (see, e.g., \cite{hairer1993,hairer1996}) will face a severe step size restriction for solving the ODE \eqref{eq:1.1}.
The implicit or implicit-explicit methods (see, e.g., \cite{alexander1977,ascher1997,hairer1996,huang2021,izzo2017,zharovsky2015}) can effectively avoid the step size restriction problem,
but they need to solve a linear or even nonlinear system at each step.
When the stiff term in the ODE \eqref{eq:1.1} is complex in form or high-dimensional, the specific implementation of implicit or implicit-explicit methods is quite challenging.

Runge-Kutta-Chebyshev (RKC) methods (see, e.g., \cite{van1980,verwer1990,verwer1996,tang2020}) are a class of explicit methods
with an extended stability domain along the negative real axis.
An important feature of the RKC methods is that the length of their stability domain along the negative real axis grows quadratically with $s$ (the number of stages).
Because the RKC methods are explicit methods, they are very easy to implement.
Meanwhile, thanks to the three-term recurrence relation, the storage demand of the RKC methods are independent of $s$ and modest.
Therefore, the RKC methods are usually very efficient for the moderately stiff ODEs arising from the
spatial semi-discretization of purely diffusive or diffusion dominated advection-diffusion problems
(very small Peclet number).
For other types of explicit stabilized methods, we refer to references \cite{abdulle2002,abdulle2001,lebedev1994,medovikov1998,martin2009,kleefeld2013,martin2016,meyer2014,o2019runge,o2015class}.

For the ODE \eqref{eq:1.1}, it is evidently more reasonable to construct a partitioned explicit stabilized method that handles stiff and non-stiff terms separately,
rather than applying an explicit stabilized method directly.
There are several partitioned explicit stabilized methods in existing literature.
Based on the RKC methods, the author of \cite{zbinden2011} constructed a class of partitioned explicit stabilized methods (called the PRKC methods).
Recently, the author of \cite{almuslimani2023} proposed another type of partitioned explicit stabilized methods (called the ARKC methods).
Based on another type of explicit stabilized methods (called the ROCK2 methods) \cite{abdulle2001},
the authors of \cite{abdulle2013} constructed a type of partitioned implicit-explicit methods (called the PIROCK methods),
which reduce to the partitioned explicit stabilized methods when extremely stiff terms are not involved.

The partitioned methods PRKC, ARKC and PIROCK  share a common characteristic: the number of stages in the Runge-Kutta (RK) method for handling $f_A$ is small and fixed.
This characteristic results in a narrow stability domain along and near the imaginary axis.
Therefore, when solving the problems that have high requirements for both the length of the method's stability domain along the negative real axis
and the width along (or near) the imaginary axis (such as the advection-diffusion equations with large Peclet number and the damped wave equations \cite{sommeijer2007,verwer2009,hundsdorfer2013}),
the computational efficiency of the partitioned methods PRKC, ARKC and PIROCK suffers significantly (see below).

In this paper, we construct a new class of partitioned explicit stabilized methods.
We treat $f_D$ with an $s$-stage RKC method and treat $f_A$ with a $4m$-stage explicit RK method,
where the parameters $s$ and $m$ can be flexibly adjusted according to the spectral radius of the Jacobian matrices of $f_D$ and $f_A$, respectively.

Compared to the partitioned methods PRKC, ARKC and PIROCK (explicit version), our new methods have the following features:

$\bullet$ \hangindent=2em \hangafter=1 the flexibility of our methods is remarkably high since both parameters $s$ and $m$ can be flexibly adjusted;

$\bullet$ \hangindent=2em \hangafter=1 our methods can demonstrate strong applicability because the stability domain of our methods can cover a rectangular area symmetric with respect to the negative real axis,
with its length extending along the negative real axis up to $0.65s^2$ and its width reaching up to $2.15m$ along the positive imaginary axis;

$\bullet$ \hangindent=2em \hangafter=1 for the problems that have high requirements for both the length of the method's stability domain along the negative real axis
and the width along (or near) the imaginary axis, our methods are more efficient than the partitioned methods PRKC, ARKC and PIROCK (explicit version);

$\bullet$ \hangindent=2em \hangafter=1 the construction idea of our methods is easily generalizable; For example,
by replacing the RKC methods with another type of explicit stabilized methods proposed in \cite{abdulle2002,abdulle2001,lebedev1994,medovikov1998,martin2009,kleefeld2013,martin2016,meyer2014},
other types of partitioned explicit stabilized methods can be easily devised (see Remark \ref{rem1}).

The rest of this paper is organized as follows. The next section reviews the RKC methods and the partitioned methods PRKC, ARKC and PIROCK.
In Section \ref{sec.3}, our new partitioned explicit stabilized methods are introduced and analyzed.
Section \ref{sec.4} presents some comparisons of several types of partitioned explicit stabilization methods.
In Section \ref{sec.5}, we introduce the control strategy for variable time step-size as well as the selection strategies for $s$ and $m$.
Finally, the numerical results are presented in Section \ref{sec.6}.

\section{Review of RKC methods and several partitioned explicit stabilized methods}
\label{sec.2}

For the convenience of subsequent description and comparison, we first briefly review several existing explicit stabilized methods (including the RKC methods, the PRKC methods, the ARKC methods and the PIROCK methods) in this section.

\subsection{The RKC methods}

Let us start by introducing the classic RKC methods (see, e.g., \cite{van1980,verwer1990}) for solving the following autonomous ODE
\begin{align}
	y'(t) &= f(y(t)), \quad y(t_0) = a \in\mathbb{R}^d, \quad t \in [t_0, T]. \label{eq:2.1}
\end{align}
Let $h=\frac{T-t_0}{N},~N\in \mathbb{Z}^+$ represent the time step-size.
Use $y_n$ to represent the numerical solution of the exact solution $y(t_n)$ with $t_n = t_0+nh$. Then, the classic $s$-stage ($s\geq 2$) RKC method can be defined as
\begin{align}
	K_0 &= y_n, \nonumber\\
	K_1 &= y_n + \tilde{u}_1 h F_0,\nonumber \\
	K_j &= u_j K_{j-1} + v_j K_{j-2} + (1 - u_j - v_j)K_0 + \tilde{u}_j h F_{j-1} + \tilde{\gamma}_j h F_0, \quad  2 \leq j\leq s, \nonumber\\
	y_{n+1} &= K_s, \quad n = 0, 1, \dots, N - 1, \ \label{eq:2.2}
\end{align}
where $y_0=y(t_0)$, the coefficients $\tilde{u}_j, u_j,v_j,\tilde{\gamma}_j\in\mathbb{R}$, and $F_j = f(K_j)$.
It can be seen from \cite{verwer1990} that the RKC method \eqref{eq:2.2} is second-order consistent for ODE system \eqref{eq:2.1} when
the coefficients $\tilde{u}_j, u_j,v_j,\tilde{\gamma}_j$ satisfy
\begin{flalign}
    \omega_0 &= 1 + \frac{\eta}{s^2}, \quad \omega_1 = \frac{T_s'(\omega_0)}{T_s''(\omega_0)}, \nonumber\\
	b_j &= \frac{T_j''(\omega_0)}{(T_j'(\omega_0))^2}, \quad  2 \leq j\leq s, \quad b_0 =b_1 = b_2, \nonumber\\
	\tilde{u}_1 &= \omega_1b_1,   \quad\tilde{u}_j = 2\omega_1\frac{b_j}{b_{j - 1}}, \quad u_j = 2\omega_0\frac{b_j}{b_{j - 1}}, \nonumber\\
    \quad v_j &= -\frac{b_j}{b_{j - 2}},   \quad \tilde{\gamma}_j=-\big{(}1-b_{j-1}T_{j-1}(\omega_0)\big{)}\tilde{u}_j, \quad  2 \leq j\leq s,   \label{eq:2.4}
\end{flalign}
where $\eta$ is an undetermined damping parameter, and $T_j(\omega_0),~1 \leq j\leq s$ denote the first kind of Chebyshev polynomials about $\omega_0$.

\indent Applying the RKC method \eqref{eq:2.2} to the following linear stability test equation
\begin{flalign}
	y' = {\lambda} y, \quad \lambda\in \mathbb{C},	\label{eq:2.5}
\end{flalign}
 it is not difficult to show that
\begin{align}
    y_{n+1}&=R_s(p)y_n, \quad K_j=R_j(p)y_n, \quad 0 \leq j\leq s, \nonumber\\
	R_j(p) &= a_j+b_jT_j(\omega_0+\omega_1p), \quad a_j=1-b_jT_{j}(\omega_0), \quad p = \lambda h.\ \label{eq:2.6}\
\end{align}
When the damping parameter $\eta$ is set to $\frac{2}{13}$, we know from \cite{verwer1990} that
\begin{align}
    |R_s(p)|\leq 1,~~p\in [-Ls^2, 0],~~L\approx 0.65.\ \label{eq:2.7}\
\end{align}
The inequality \eqref{eq:2.7} indicates that the length of the stability domain along the negative real axis grows quadratically with $s$ for the RKC method \eqref{eq:2.2}.
For example, taking $s = 5, 15, 50$ in turn, we can obtain Figure \ref{fig:RKCwdy}, where the horizontal axis represents the real axis and the vertical axis is the imaginary axis.
\begin{figure}
	\centering
	\begin{minipage}[b]{0.30\textwidth}
		\includegraphics[width=\textwidth]{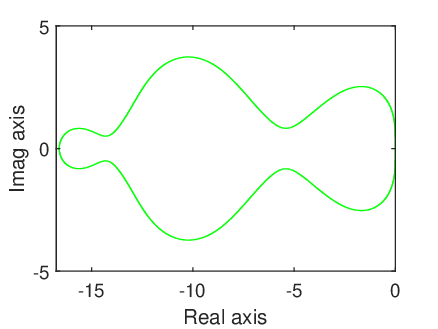} 	
	\end{minipage}
	\hfill 
	\begin{minipage}[b]{0.30\textwidth}
		\includegraphics[width=\textwidth]{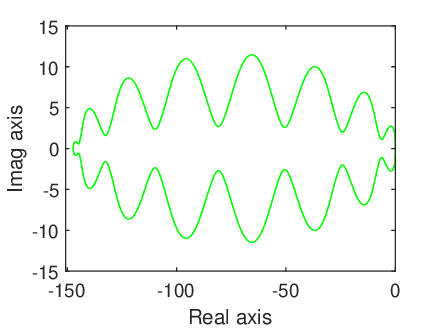}  
	\end{minipage}
	\hfill 
	\begin{minipage}[b]{0.30\textwidth}
		\includegraphics[width=\textwidth]{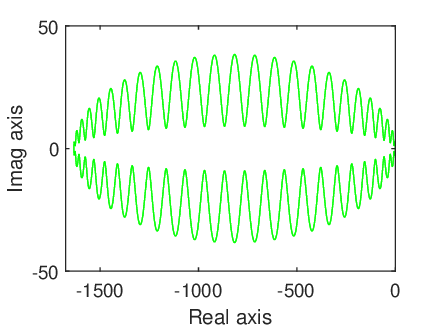}  
	\end{minipage}
	\caption{ Stability domains of the RKC method \eqref{eq:2.2}. $s = 5$ (left), $s = 15$ (middle), $s =50$ (right).}
	\label{fig:RKCwdy}
\end{figure}

\subsection{The PRKC methods}

In view of the advantage of RKC method \eqref{eq:2.2} in the linear stability, the author in reference \cite{zbinden2011} extended it to the ODE system \eqref{eq:1.1} and constructed the PRKC methods.

For the ODE system \eqref{eq:1.1}, the $s$-stage PRKC method can be defined as
\begin{flalign}
	K_{-1} &= y_n, \nonumber\\
	K_0 &= K_{-1} + \alpha_0 h f_{A}(K_{-1}),\nonumber \\
	K_1 &= K_0 + \tilde{u}_1 h f_{D}({K_0}),\nonumber \\
	K_j &= u_j K_{j-1} + v_j K_{j-2} + (1 - u_j - v_j)K_0 + \tilde{u}_j h f_D(K_{j-1}) + \tilde{\gamma}_j h f_D(K_0), \nonumber \\
	& ~~\quad j = 2, \dots, s - 1,\nonumber \\
	K_s &= u_s K_{s-1} + v_s K_{s-2} + (1 - u_s - v_s)K_0 + \tilde{u}_s h f_D(K_{s-1}) + \tilde{\gamma}_s h f_D(K_0) \nonumber \\
	& \quad + \alpha_1 h f_{A}(K_{-1}) + \alpha_2 h f_{A}(K_{0}) + \alpha_3 h f_{A}(K_{s-1}),\nonumber \\
	K_{s + 1} &= u_s K_{s-1} + v_s K_{s-2} + (1 - u_s - v_s)K_0 + \tilde{u}_s h f_D(K_{s-1}) + \tilde{\gamma}_s h f_D(K_0) \nonumber \\
	& \quad + \alpha_4 h f_{A}(K_{-1}) + \alpha_5 h f_{A}(K_{0}) + \alpha_6 h f_{A}(K_{s-1}) + \alpha_7 h f_{A}(K_{s}), \nonumber\\
	y_{n + 1} &= K_{s + 1}, \quad n = 0, 1, \dots, N - 1, \label{eq:2.8}
\end{flalign}
where $\alpha_0, \alpha_1, \cdots, \alpha_7\in\mathbb{R}$, and the coefficients $\tilde{u}_j, u_j,v_j,\tilde{\gamma}_j$ are defined by \eqref{eq:2.4}.
If $f_A\equiv0$, the method \eqref{eq:2.8} will degenerate into the standard RKC method.
If $f_D\equiv0$, the method \eqref{eq:2.8} will degenerate into the $3$-stage RK method
\begin{flalign}
	H_{1} &= y_n, \nonumber\\
	H_2 &= y_n + \alpha_0 h f_{A}(H_1),\nonumber \\
    H_3 &= y_n + (\alpha_0+\alpha_1) h f_{A}(H_1)+ (\alpha_2+\alpha_3) h f_{A}(H_2),\nonumber \\
    y_{n+1} &= y_n + (\alpha_0+\alpha_4) h f_{A}(H_1)+ (\alpha_5+\alpha_6) h f_{A}(H_2)+\alpha_7hf_{A}(H_3). \label{eq:2.9}
\end{flalign}
We know from \cite{zbinden2011} that the RK method \eqref{eq:2.9} is third-order consistent when
the coefficients $\alpha_0, \alpha_1, \cdots, \alpha_7$ satisfy
\begin{flalign}
	\alpha_0 &= \frac{1}{2},  \quad \alpha_1 = -\frac{1}{2} + r(3 - 4r),
	\quad  \alpha_2 = 2r(2r - 1) - \alpha_3, \nonumber \\
	\alpha_4 &= \frac{1 - 3r}{6r},\quad
	\alpha_5 = \frac{1 + 3r(1 - 2r) + 4c_{s - 1}r(3r - 2)}{6c_{s - 1}r(2r - 1)},\quad
	\alpha_6 = \frac{3r(2r - 1) - 1}{6c_{s - 1}r(2r - 1)}, \nonumber\\
	\alpha_7 &= \frac{1}{6r(2r - 1)}, \label{eq:2.10}
\end{flalign}
where $r\notin\{0, 1/2\}$, $\alpha_3$ are free parameters,
and $c_{s-1}$ can be obtained through recursive calculation as
\begin{align}	
	c_0 &= 0, \quad c_1 = \tilde{u}_1,\nonumber \\
	c_j &= u_j c_{j - 1} + v_j c_{j - 2} + \tilde{u}_j + \tilde{\gamma}_j, \quad 2 \leq j \leq s. \label{eq:2.11}
\end{align}

\indent Applying the PRKC method \eqref{eq:2.8} to the test equation
\begin{flalign}
	y' = {\lambda_1} y+i{\lambda_2} y, \quad \lambda_1, \lambda_2\in \mathbb{R},	\label{eq:2.12}
\end{flalign}
it is not difficult to show that
\begin{align}
    y_{n+1}&=\tilde{R}_s(p,q)y_n,  ~~~~p=\lambda_1 h,~q=\lambda_2 h,\nonumber\\
	\tilde{R}_s(p,q) &= R_s(p)\big{(}1+(\alpha_0+\alpha_7)iq+\alpha_0\alpha_7(iq)^2\big{)}  \nonumber\\
                        &~~~+R_{s-1}(p)\big{(}\alpha_6iq+(\alpha_0\alpha_6+\alpha_3\alpha_7)(iq)^2+\alpha_0\alpha_3\alpha_7(iq)^3\big{)} \nonumber\\
                        &~~~+(\alpha_4+\alpha_5)iq+(\alpha_0\alpha_5+(\alpha_1+\alpha_2)\alpha_7)(iq)^2+\alpha_0\alpha_2\alpha_7(iq)^3,
  \label{eq:2.13}\
\end{align}
where $R_{s-1}(p), R_s(p)$ are defined by \eqref{eq:2.6}.

\indent We know from \cite{zbinden2011} that $r=1$ and $\alpha_3=0$ are a suitable set of parameters, then we can obtain Figure \ref{fig:PRKCwdy} by
combining \eqref{eq:2.10} and \eqref{eq:2.13}.
It is not difficult to see from Figure \ref{fig:PRKCwdy} that the stability domain of
the PRKC method \eqref{eq:2.8} can cover a narrow rectangular domain.
According to \cite{zbinden2011}, the length of the rectangular domain can reach to $0.65s^2$, and its width is about twice $1.73$.
This means that the PRKC method \eqref{eq:2.8} inherits
 the stability advantage of the RKC method \eqref{eq:2.2}.
Note that the PRKC method \eqref{eq:2.8} requires only $4$ evaluations for $f_A$ at each step. Since $4$ is usually much smaller than $s$, for the ODE system \eqref{eq:1.1}, the calculate cost of PRKC method usually less than that of the standard RKC method at each step.
\begin{figure}
	\centering
	\begin{minipage}[b]{0.30\textwidth}
		\includegraphics[width=\textwidth]{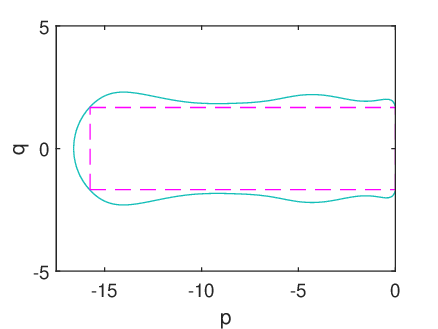} 	
	\end{minipage}
	\hfill 
	\begin{minipage}[b]{0.30\textwidth}
		\includegraphics[width=\textwidth]{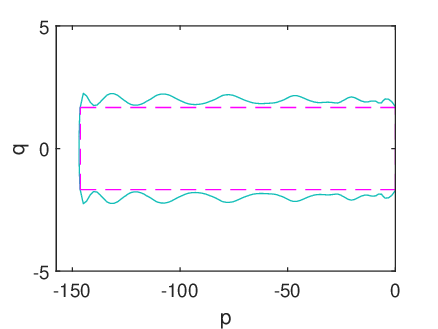}  
	\end{minipage}
	\hfill 
	\begin{minipage}[b]{0.30\textwidth}
		\includegraphics[width=\textwidth]{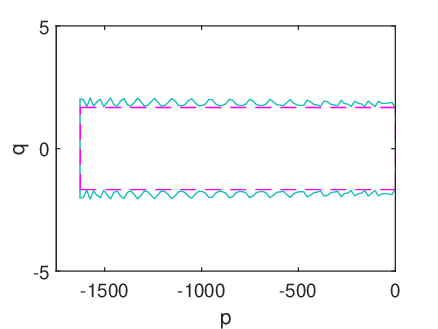}  
	\end{minipage}
	\caption{ Stability domains (internal area of the blue solid line) of the PRKC  method \eqref{eq:2.8} and the interior rectangulars (pink dashed line). $s = 5$ (left), $s = 15$ (middle), $s =50$ (right).}
	\label{fig:PRKCwdy}
\end{figure}

\subsection{The ARKC methods}

Based on the RKC methods, the author of \cite{almuslimani2023} recently proposed another class of partitioned explicit stabilized methods called the ARKC methods.

For the ODE system \eqref{eq:1.1}, the $s$-stage ARKC method can be defined as
\begin{flalign}
	G &= hf_A \left(y_n + \frac{h}{2} f_A \left(y_n + \frac{\omega_1}{2} h f_D(y_n)\right) + \frac{h}{2} f_D(y_n)\right) \nonumber \\
	&\quad + hf_D\left(y_n + \frac{\omega_1 - 1}{2} h f_A(y_n)\right)- hf_D(y_n), \nonumber \\
	K_0 &= y_n + \frac{\omega_1}{2} G, \nonumber \\
	K_1 &= K_0 + \tilde{u}_1 h f_D(y_n) + \hat{\alpha} G,  \nonumber \\
    K_j &= u_j K_{j-1} + v_j K_{j-2} + (1 - u_j - v_j)K_0 + \tilde{u}_j h \big{(}f_D(K_{j-1})-f_D(K_{0})+f_D(y_n)\big{)}  \nonumber \\
	& \quad + \tilde{\gamma}_j h f_D(y_n), ~~j = 2, \dots, s,\nonumber \\
	y_{n+1} &= K_s, \label{eq:2.14}
\end{flalign}
where $\hat{\alpha} = (1-\frac{\omega_1}{2})b_1s\omega_1$ and $\tilde{u}_j, u_j, v_j, \tilde{\gamma}_j, b_1, \omega_1$ are defined by \eqref{eq:2.4}.
The method \eqref{eq:2.14} also degenerates into the standard RKC method when $f_A\equiv0$.
If $f_D\equiv0$, the method \eqref{eq:2.14} will degenerate into the $2$-stage second-order RK method
\begin{flalign}
	H_{1} &= y_n, \nonumber\\
	H_2 &= y_n + \frac{1}{2} h f_{A}(H_1),\nonumber \\
    y_{n+1} &= y_n +hf_{A}(H_2). \label{eq:2.15}
\end{flalign}

According to reference \cite{almuslimani2023}, applying the method \eqref{eq:2.14} to test equation \eqref{eq:2.12} yields
\begin{align}
    y_{n+1}&=\hat{R}_s(p,q)y_n,  ~~~~p=\lambda_1 h,~q=\lambda_2 h,\nonumber\\
	\hat{R}_s(p,q) &= R_s(p)+\Big{(}\frac{\omega_1}{2}+\big{(}1-\frac{\omega_1}{2}\big{)}\frac{U_{s-1}(\omega_0+\omega_1p)}{U_{s-1}(\omega_0)}\Big{)} \Big{(}1+\frac{\omega_1}{2}p\Big{)}\Big{(}iq+\frac{1}{2}(iq)^2\Big{)},               \label{eq:2.16}\
\end{align}
where $U_{s-1}(\cdot)$ represents the second type of Chebyshev polynomial of degree $s-1$.
Based on \eqref{eq:2.16}, we can obtain Figure \ref{fig:ARKCwdy}.
It is not difficult to see from Figure \ref{fig:ARKCwdy} that the stability domain of
the ARKC method \eqref{eq:2.14} can cover a suitable elliptical domain.
For the ARKC method \eqref{eq:2.14}, the damping parameter $\eta$ no longer takes a fixed value, but changes with the variation of $s$ to make the covered elliptical domain more optimal (Please refer to reference \cite{almuslimani2023} for details).
\begin{figure}
	\centering
	\begin{minipage}[b]{0.30\textwidth}
		\includegraphics[width=\textwidth]{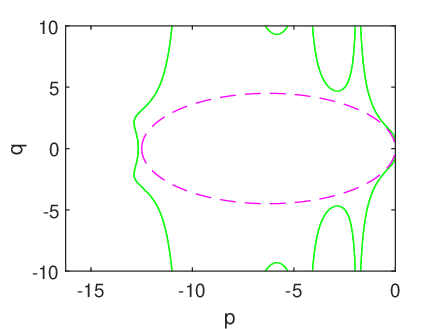} 	
	\end{minipage}
	\hfill 
	\begin{minipage}[b]{0.30\textwidth}
		\includegraphics[width=\textwidth]{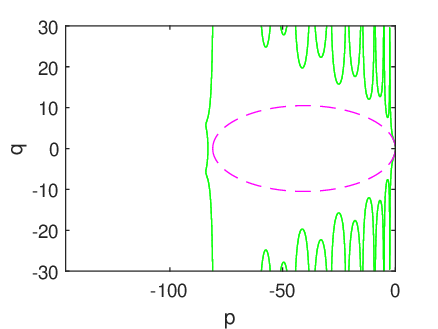}  
	\end{minipage}
	\hfill 
	\begin{minipage}[b]{0.30\textwidth}
		\includegraphics[width=\textwidth]{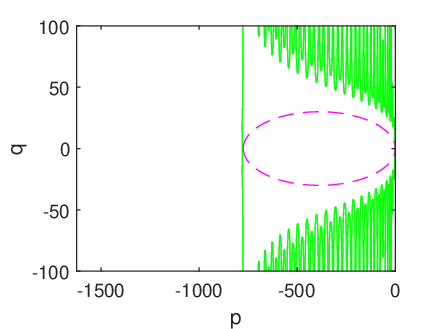}  
	\end{minipage}
	\caption{ Stability domains (internal area of the green solid line) of the ARKC method \eqref{eq:2.14} and the interior ellipses (pink dashed line). $s = 5,\eta = 4$ (left), $s = 15,\eta = 9$ (middle), $s =50,\eta = 13.5$ (right).}
	\label{fig:ARKCwdy}
\end{figure}

\subsection{The PIROCK methods}

Both the PRKC method \eqref{eq:2.8} and the ARKC method \eqref{eq:2.14} are constructed based on the RKC method \eqref{eq:2.2}.
In this subsection, we introduce the PIROCK methods \cite{abdulle2013},
which are constructed based on another type of second-order explicit stabilized methods called the ROCK2 methods \cite{abdulle2001}.
Because the ODE system \eqref{eq:1.1} does not involve extremely stiff terms, we mainly introduce the explicit version of the PIROCK methods after degradation in this subsection.

According to reference \cite{abdulle2001}, the $s$-stage ROCK2 method for the ODE system \eqref{eq:2.1} can be defined as
\begin{align}
	K_0 &= y_n, \quad K_1 = y_n + \alpha\mu_1 h f(y_n), \nonumber\\
	K_j &= \alpha\mu_j h f(K_{j - 1}) - \nu_j K_{j - 1} - \kappa_j K_{j - 2}, \quad j = 2, \dots, s - 2, \nonumber\\
	K^*_{s - 1} &= K_{s - 2} + \sigma_{\alpha} h f(K_{s - 2}),~~
	K^*_s = K^*_{s - 1} + \sigma_{\alpha} h f(K^*_{s - 1}),\nonumber \\
	y_{n+1} &= K^*_s - \sigma_{\alpha}(1 - \tau_{\alpha} / \sigma_{\alpha}^2)(h f(K^*_{s - 1}) - h f(K_{s - 2})), \label{eq:2.17}
\end{align}
where $\alpha\geq 1$ is the damping parameter, and $\sigma_{\alpha}, \tau_{\alpha}$ satisfy
\[
	\sigma_{\alpha}=\frac{1-\alpha}{2}+\alpha \sigma,~~\tau_{\alpha}=\frac{(\alpha-1)^2}{2}+2\alpha(1-\alpha)\sigma+\alpha^2\tau.
\]
For the ROCK2 method \eqref{eq:2.17}, the coefficients $\mu_j, \nu_j, \kappa_j$ and the parameters $\sigma, \tau$
no longer have analytical expressions like the coefficients in \eqref{eq:2.4}, but need to be calculated through specific numerical algorithms (Please refer to \cite{abdulle2001} for details).
Applying the ROCK2 method \eqref{eq:2.17} to the test equation \eqref{eq:2.5} yields
\begin{align}
    y_{n+1}&=\check{R}_s(\alpha p)y_n, \quad K_j=P_j(\alpha p)y_n,  ~~p=\lambda h,\nonumber\\
    P_0(\alpha p)&=1,~~P_1(\alpha p)=1+\alpha\mu_1 p, \nonumber\\
    P_j(\alpha p)&=(\alpha\mu_j p - \nu_j) P_{j-1}(\alpha p) - \kappa_j P_{j-2}(\alpha p), ~~2 \leq j\leq s-2, \nonumber\\
	\check{R}_s(\alpha p) &= (1+2\sigma_{\alpha}p+\tau_{\alpha}p^2)P_{s-2}(\alpha p).  \label{eq:2.18}\
\end{align}
When the damping parameter $\alpha=1$, we know from \cite{abdulle2001} that
$|\check{R}_s(\alpha p)|\leq 1,~p\in [-Ls^2, 0],~L\approx 0.81$.
This means that the length of the stability domain along the negative real axis is almost optimal for the second-order ROCK2 method \eqref{eq:2.17}.

For the ODE system \eqref{eq:1.1}, the PIROCK methods proposed in \cite{abdulle2013} can be defined as
\begin{flalign}
	K_0 &= y_n, \quad K_1 = y_n + \alpha\mu_1 h f_D(y_n), \nonumber\\
	K_j &= \alpha\mu_j h f_D(K_{j - 1}) - \nu_j K_{j - 1} - \kappa_j K_{j - 2}, \quad j = 2, \dots, s - 2 + \ell ~(\ell = 1 \text{ or } 2),\nonumber\\
	K^*_{s - 1} &= K_{s - 2} + \sigma_{\alpha} h f_D(K_{s - 2}),~~
	K^*_s = K^*_{s - 1} + \sigma_{\alpha} h f_D(K^*_{s - 1}),\nonumber \\
	K_{s + 1} &= K_{s - 2 + \ell}, ~~
	K_{s + 2} = K_{s + 1} + (1 - 2\gamma)h f_A(K_{s + 1}), \nonumber \\
	K_{s + 3} &= K_{s + 1} + \frac{1}{3}h f_A(K_{s + 1}),~~
	K_{s + 4} = K_{s + 1} + \frac{2\beta}{3}h f_D(K_{s + 1}) + \frac{2}{3}h  f_A(K_{s + 3}),\nonumber\\
	y_{n+1} &= K^*_s - \sigma_{\alpha}(1 - \tau_{\alpha} / \sigma_{\alpha}^2)(h f_D(K^*_{s - 1}) - h f_D(K_{s - 2})),\nonumber \\
	&\quad + \frac{1}{4}h f_A(K_{s + 1}) + \frac{3}{4}h f_A(K_{s + 4})  + \frac{1}{2 - 4\gamma}(h f_D(K_{s + 2}) - h f_D(K_{s + 1})), \label{eq:2.19}
\end{flalign}
where $\gamma=1-\sqrt{2}/2$ and $\beta=1-2\alpha P'_{s-2+\ell}(0)$ (Please refer to \eqref{eq:2.18} for the definition of polynomial $P_{s-2+\ell}$).
The method \eqref{eq:2.19} degenerates into the standard ROCK2 method when $f_A\equiv0$.
If $f_D\equiv0$, the method \eqref{eq:2.19} will degenerate into the $3$-stage third-order RK method
\begin{flalign}
	H_{1} &= y_n, \nonumber\\
	H_2 &= y_n + \frac{1}{3} h f_{A}(H_1),\nonumber \\
    H_3 &= y_n + \frac{2}{3} h f_{A}(H_2),\nonumber \\
    y_{n+1} &= y_n +\frac{1}{4}hf_{A}(H_1)+\frac{3}{4}hf_{A}(H_3). \label{eq:2.20}
\end{flalign}
Applying the method \eqref{eq:2.19} to test equation \eqref{eq:2.12} yields
\begin{align}
    y_{n+1}&=\bar{R}_s(\alpha p,q)y_n,  ~~~~p=\lambda_1 h,~q=\lambda_2 h,\nonumber\\
	\bar{R}_s(\alpha p,q) &= \check{R}_s(\alpha p)
+P_{s-2+\ell}(\alpha p)\Big{(}iq+\frac{1}{2}(iq)^2+\frac{1}{6}(iq)^3+\frac{1}{2}(\beta+1)p(iq)\Big{)}.              \label{eq:2.21}\
\end{align}
In reference \cite{abdulle2013}, two sets of choices are provided for the parameters $\alpha$ and $\ell$.
The first set of choices is $\alpha=1, \ell=2$, and the second set of choices is $\alpha=1/(2P'_{s-1}(0)), \ell=1$.
Specifically, taking the second set of parameters,
based on \eqref{eq:2.21}, we can obtain Figure \ref{fig:PIROCKwdy}.
It is not difficult to see from Figure \ref{fig:PIROCKwdy} that the stability domain of
the PIROCK method \eqref{eq:2.19} can also cover a suitable elliptical domain.
\begin{figure}
	\centering
	\begin{minipage}[b]{0.30\textwidth}
		\includegraphics[width=\textwidth]{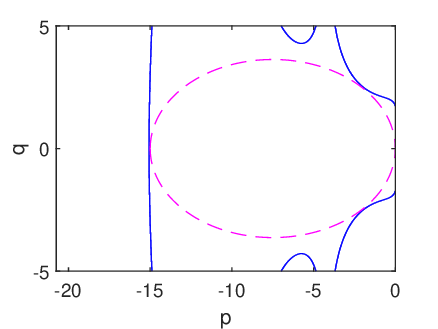} 	
	\end{minipage}
	\hfill 
	\begin{minipage}[b]{0.30\textwidth}
		\includegraphics[width=\textwidth]{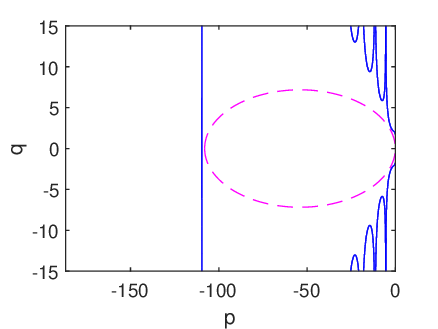}  
	\end{minipage}
	\hfill 
	\begin{minipage}[b]{0.30\textwidth}
		\includegraphics[width=\textwidth]{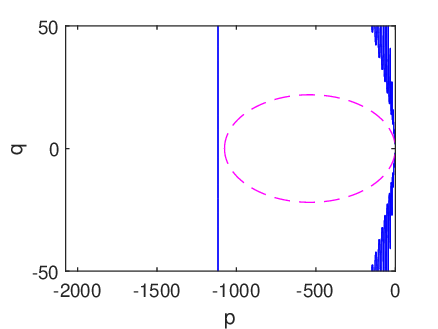}  
	\end{minipage}
	\caption{ Stability domains (internal area of the blue solid line) of the PIROCK method \eqref{eq:2.19} and the interior ellipses (pink dashed line). $s = 5$ (left), $s = 15$ (middle), $s =50$ (right).}
	\label{fig:PIROCKwdy}
\end{figure}

\section{A new class of partitioned explicit stabilized methods}
\label{sec.3}

For the three types of partitioned explicit stabilized methods reviewed in the previous section, the stage of the RK method for handling $f_A$ is small and fixed.
This results in significant limitation on the width of their stability domains along and near the imaginary axis (see Figure \ref{fig:PRKCwdy}-Figure \ref{fig:PIROCKwdy}).

In this section, we will propose a new class of partitioned explicit stabilized methods.
An important feature of our partitioned methods is that the stage of the RK method for handling $f_A$ can be flexibly adjusted according to the needs of the problem.
This feature greatly enhances the flexibility and applicability of our methods.

For the ODE system \eqref{eq:1.1}, our partitioned methods can be defined as
\begin{flalign}
	\hat{K}_0 &= y_n,~~\hat{K}_i=\hat{K}_{i-1}+\frac{1}{2m}hf_A(\hat{K}_{i-1}),~~i=1,2,\dots, m,\nonumber \\
	K_0 &= \hat{K}_m, \nonumber \\
	K_1 &= K_0 + \tilde{u}_1 h f_D(K_0),  \nonumber \\
    K_j &= u_j K_{j-1} + v_j K_{j-2} + (1 - u_j - v_j)K_0 + \tilde{u}_j h f_D(K_{j-1})
    + \tilde{\gamma}_j h f_D(K_0), ~~j = 2, \dots, s,\nonumber \\
    K_{s+3i-2}&=K_{s+3i-3}+ \frac{1}{6m}h f_A(K_{s+3i-3}),  \nonumber \\
    K_{s+3i-1}&=K_{s+3i-3}- \frac{1}{6m}h f_A(K_{s+3i-2}),  \nonumber \\
    K_{s+3i}&=K_{s+3i-3}+\frac{2}{m}h f_A(K_{s+3i-3})-\frac{3}{2m}h f_A(K_{s+3i-1}), ~~i=1,2,\dots, m, \nonumber \\
	y_{n+1} &= K_{s+3m}, \label{eq:3.1}
\end{flalign}
where the coefficients $\tilde{u}_j, u_j, v_j, \tilde{\gamma}_j$ are defined by \eqref{eq:2.4}, and $m$ is a positive integer to be determined.
The method \eqref{eq:3.1} degenerates into the standard RKC method when $f_A\equiv0$.
If $f_D\equiv0$, the method \eqref{eq:3.1} will degenerate into the $4m$-stage RK method
\begin{flalign}
	H_0 &= y_n,~~H_j=H_{j-1}+\frac{1}{2m}hf_A(H_{j-1}),~~j=1,2,\dots, m,\nonumber \\
    H_{m+3i-2}&=H_{m+3i-3}+ \frac{1}{6m}h f_A(H_{m+3i-3}),  \nonumber \\
    H_{m+3i-1}&=H_{m+3i-3}- \frac{1}{6m}h f_A(H_{m+3i-2}),  \nonumber \\
    H_{m+3i}&=H_{m+3i-3}+\frac{2}{m}h f_A(H_{m+3i-3})-\frac{3}{2m}h f_A(H_{m+3i-1}), ~~i=1,2,\dots, m, \nonumber \\
	y_{n+1} &= H_{4m}. \label{eq:3.2}
\end{flalign}
Specifically, if we take $m=1$, the method \eqref{eq:3.2} can be simplified into the following $4$-stage third order RK method
\begin{flalign}
	H_{1} &= y_n, \nonumber \\
	H_{2} &=y_n+\frac{1}{2}h f_A(H_{1}), \nonumber \\
	H_{3} &=y_n+\frac{1}{2}h f_A(H_{1})+ \frac{1}{6}h f_A(H_{2}), \nonumber \\
	H_{4} &=y_n+\frac{1}{2}h f_A(H_{1})- \frac{1}{6}h f_A(H_{3}), \nonumber \\
	y_{n+1} &=y_n+\frac{1}{2}h f_A(H_{1})+2h f_A(H_{2}) -\frac{3}{2}h f_A(H_{4}). \label{eq:3.3}
\end{flalign}

Applying our partitioned method \eqref{eq:3.1} to test equation \eqref{eq:2.12} yields
\begin{align}
    y_{n+1}&=\ddot{R}_{s,m}(p,q)y_n,  ~~~~p=\lambda_1 h,~q=\lambda_2 h,\nonumber\\
	\ddot{R}_{s,m}(p,q) &= \big{(}1+\frac{1}{2m}(iq)\big{)}^mR_s(p)\big{(}1+\frac{1}{2m}(iq)
                             +\frac{1}{4m^2}(iq)^2+\frac{1}{24m^3}(iq)^3\big{)}^m,              \label{eq:3.4}\
\end{align}
where $R_s(p)$ is the stability function of RKC method \eqref{eq:2.2}.
Based on the stability function $\ddot{R}_{s,m}(p,q)$ in \eqref{eq:3.4}, we can obtain the following stability result.

\begin{theorem}
\label{thm1}
The stability domain of the partitioned method \eqref{eq:3.1} can cover the following rectangular domain
\begin{align}
    S(s,m)=\Big{\{}(p,q)\in \mathbb{R}^2 \Big{|} -Ls^2\leq p\leq 0, ~-2.15m\leq q\leq 2.15m \Big{\}}, \label{eq:3.5}\
\end{align}
where $L\approx 0.65$ if the damping parameter $\eta=\frac{2}{13}$.
\end{theorem}

\begin{proof}
To prove the above conclusion, we only need to prove
\begin{align}
    \big{|}\ddot{R}_{s,m}(p,q)\big{|}\leq 1,~~(p,q)\in S(s,m). \label{eq:3.6}\
\end{align}
Based on \eqref{eq:3.4}, we have
\begin{align}
    \big{|}\ddot{R}_{s,m}(p,q)\big{|}=\Big{|}R_s(p)\Big{|}\Big{|}\big{(}1+\frac{1}{2m}(iq)\big{)}\big{(}1+\frac{1}{2m}(iq)
                             +\frac{1}{4m^2}(iq)^2+\frac{1}{24m^3}(iq)^3\big{)}\Big{|}^m.  \label{eq:3.7}\
\end{align}
Therefore, if we can show that
\begin{align}
    \Big{|}R_s(p)\Big{|}&\leq 1,~~-Ls^2\leq p\leq 0, \nonumber\\
    \Big{|}\big{(}1+\frac{1}{2m}(iq)\big{)}\big{(}1+\frac{1}{2m}(iq)
                             +\frac{1}{4m^2}(iq)^2+\frac{1}{24m^3}(iq)^3\big{)}\Big{|}&\leq 1, ~~-2.15m\leq q\leq 2.15m,  \label{eq:3.8}\
\end{align}
then conclusion \eqref{eq:3.6} holds.
According to inequality \eqref{eq:2.7}, we know that the first inequality in \eqref{eq:3.8} holds.

Next, we will prove the second inequality in \eqref{eq:3.8}.
In fact, we can directly verify that
\begin{align}
    \Big{|}\big{(}1+\frac{1}{2}(iq)\big{)}\big{(}1+\frac{1}{2}(iq)
                             +\frac{1}{4^2}(iq)^2+\frac{1}{24^3}(iq)^3\big{)}\Big{|}&\leq 1, ~~-2.15\leq q\leq 2.15.  \label{eq:3.9}\
\end{align}
Therefore, we have
\begin{align}
    &\Big{|}\big{(}1+\frac{1}{2m}(iq)\big{)}\big{(}1+\frac{1}{2m}(iq)
                             +\frac{1}{4m^2}(iq)^2+\frac{1}{24m^3}(iq)^3\big{)}\Big{|}  \nonumber \\
      =&\Big{|}\big{(}1+\frac{1}{2}(i\frac{q}{m})\big{)}\big{(}1+\frac{1}{2}(i\frac{q}{m})
                             +\frac{1}{4}(i\frac{q}{m})^2+\frac{1}{24}(i\frac{q}{m})^3\big{)}\Big{|}
      \leq 1,                      ~~-2.15m\leq q\leq 2.15m.  \label{eq:3.10}\
\end{align}
The proof is complete.
\end{proof}
By taking some specific values for $s$ and $m$, we can obtain Figures \ref{fig:NPRKCwdy1} and \ref{fig:NPRKCwdy2},
which can more intuitively explain the conclusion of Theorem \ref{thm1}.
\begin{figure}
	\centering
	\begin{minipage}[b]{0.30\textwidth}
		\includegraphics[width=\textwidth]{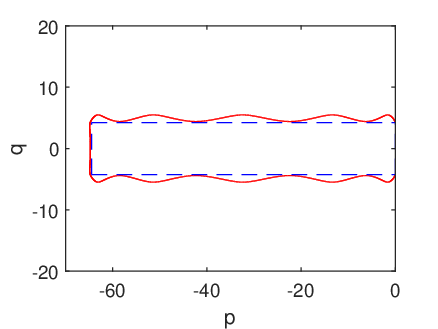} 	
	\end{minipage}
	\hfill 
	\begin{minipage}[b]{0.30\textwidth}
		\includegraphics[width=\textwidth]{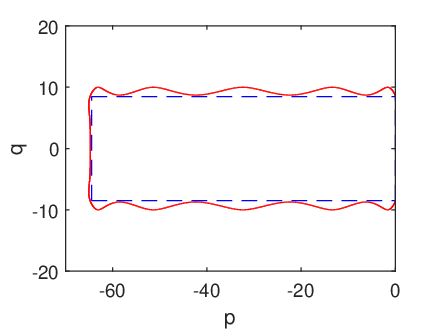}  
	\end{minipage}
	\hfill 
	\begin{minipage}[b]{0.30\textwidth}
		\includegraphics[width=\textwidth]{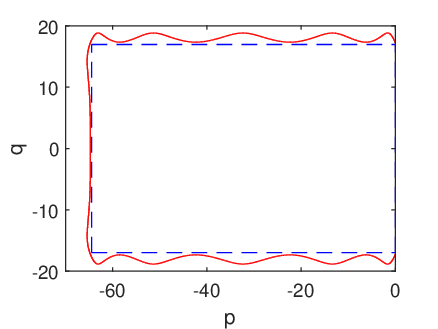}  
	\end{minipage}
	\caption{ Stability domains (internal area of the red solid line) of our method \eqref{eq:3.1} and the interior rectangulars (blue dashed line). $s = 10, m=2$ (left), $s = 10, m=4$ (middle), $s =10, m=8$ (right).}
	\label{fig:NPRKCwdy1}
\end{figure}

\begin{figure}
	\centering
	\begin{minipage}[b]{0.30\textwidth}
		\includegraphics[width=\textwidth]{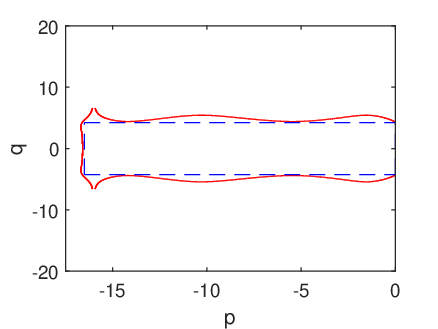} 	
	\end{minipage}
	\hfill 
	\begin{minipage}[b]{0.30\textwidth}
		\includegraphics[width=\textwidth]{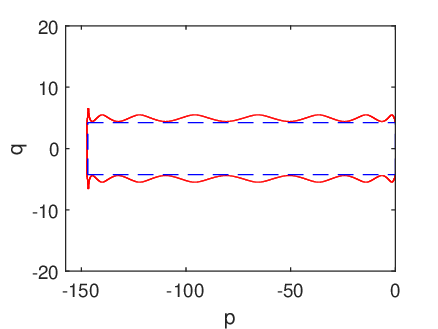}  
	\end{minipage}
	\hfill 
	\begin{minipage}[b]{0.30\textwidth}
		\includegraphics[width=\textwidth]{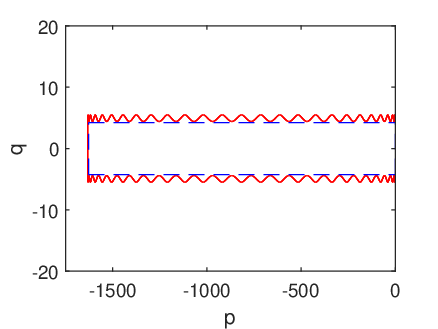}  
	\end{minipage}
	\caption{ Stability domains (internal area of the red solid line) of our method \eqref{eq:3.1} and the interior rectangulars (blue dashed line). $s = 5, m=2$ (left), $s = 15, m=2$ (middle), $s =50, m=2$ (right).}
	\label{fig:NPRKCwdy2}
\end{figure}

We next discuss the convergence result of our partitioned method \eqref{eq:3.1}.

\begin{theorem}
\label{thm2}
The partitioned method \eqref{eq:3.1} is second-order consistent for the ODE system \eqref{eq:1.1}.
\end{theorem}

\begin{proof}
Because the consistent conditions for nonlinear problems and linear problems within second order are the same, we only need to prove that
\begin{align}
    \ddot{R}_{s,m}(p,q)=1+p+(iq)+\frac{1}{2}p^2+\frac{1}{2}(iq)^2+\frac{1}{2}p(iq)+\frac{1}{2}(iq)p+\mathcal{O}(h^3). \label{eq:3.11}\
\end{align}
Because the RKC method \eqref{eq:2.2} is second-order consistent for the ODE system \eqref{eq:2.1},
we have
\begin{align}
    \big{(}1+\frac{1}{2m}(iq)\big{)}^m &=1+\frac{1}{2}(iq)+\frac{m(m-1)}{2}\frac{1}{4m^2}(iq)^2+\mathcal{O}(q^3),  \nonumber \\
    R_s(p) &=1+p+\frac{1}{2}p^2+\mathcal{O}(p^3),  \nonumber \\
    \big{(}1+\frac{1}{2m}(iq)+\frac{1}{4m^2}(iq)^2+\frac{1}{24m^3}(iq)^3\big{)}^m &=1+\frac{1}{2}(iq)+(\frac{m(m-1)}{2}+m)\frac{1}{4m^2}(iq)^2  \nonumber \\
    &~~~~+\mathcal{O}(q^3).  \label{eq:3.12}\
\end{align}
Note that $p=\mathcal{O}(h),~q=\mathcal{O}(h)$, therefore the conclusion \eqref{eq:3.11} follows from \eqref{eq:3.4} and \eqref{eq:3.12}.
The proof is complete.
\end{proof}

\begin{remark}
\label{rem1}
From the proof process of Theorem \ref{thm2}, it can be seen that our partitioned method \eqref{eq:3.1} is still second-order consistent
if we replace the RKC methods with other second-order explicit stabilized methods proposed in \cite{tang2020,abdulle2001,martin2009,kleefeld2013,meyer2014}.
In particular, if we replace the RKC methods with the ROCK2 methods, the coefficient $L$ in \eqref{eq:3.5} can reach to 0.81 (almost optimal value).
These extended analyses follow procedures analogous to those in the text, they are left to the interested readers.
\end{remark}

\section{Comparison of several types of partitioned explicit stabilized methods}
\label{sec.4}

In this section, we will compare our partitioned method \eqref{eq:3.1} with the three types of partitioned methods reviewed in the Section \ref{sec.2}.

In the case of $m=1$, the computational cost of our partitioned method \eqref{eq:3.1} in each step is the same as that of the PRKC method \eqref{eq:2.8},
and our method's stability domain is slightly better than that of PRKC method (see Figure \ref{fig:Com1}).
More importantly, the width of the stability domain along the imaginary axis can be expanded proportionally with $m$ for our partitioned method \eqref{eq:3.1}.
Therefore, compared with the PRKC method \eqref{eq:2.8}, our partitioned method \eqref{eq:3.1} has significant advantages in flexibility and applicability.
\begin{figure}[h]
	\centering
	\includegraphics[width=0.7\linewidth]{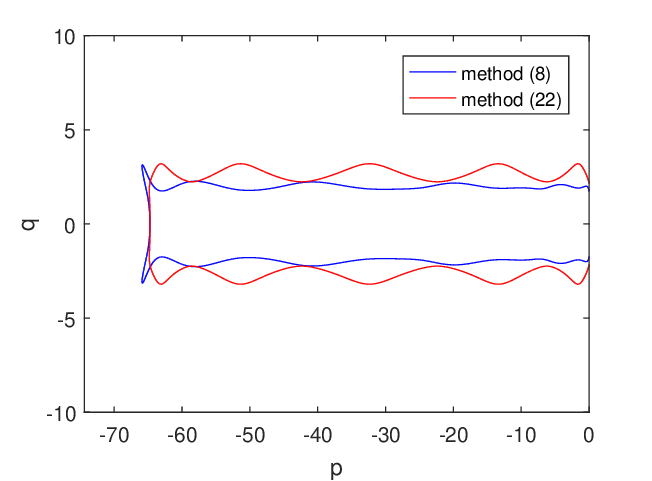}
	\caption{Comparison of the stability domains of the methods \eqref{eq:2.8} and \eqref{eq:3.1} for $s = 10, m=1$.}
	\label{fig:Com1}
\end{figure}

Next, we will focus on comparing partitioned methods \eqref{eq:2.14}, \eqref{eq:2.19} and \eqref{eq:3.1}. To make the comparison more intuitive, we will center the discussion around two concrete problems.

\subsection{Linear 1-dimensional advection-diffusion problem}

Consider the following linear advection-diffusion problem with periodic
boundary condition
\begin{flalign}
        w_t+Aw_x &=Dw_{xx}, \nonumber \\
        w(x,0) &=w_0(x), ~~x\in[0,1], \nonumber \\
        w(0, t) &= w(1,t), ~~t\in[0,T], \label{eq:4.1}
\end{flalign}
where $A$ and $D$ are two constant coefficients.
We discretize spatial variable $x$ with the grid size $h_x = \frac{1}{N}, N\in\mathbb{Z}^+$.
Using the second-order central difference formula to discretize the advection and diffusion terms yields
 \begin{flalign}
 	w_j^{\prime}(t) =  D\frac{w_{j-1}(t) - 2w_{j}(t) + w_{j+1}(t) }{h_x^2}+A  \frac{w_{j-1}(t) -w_{j+1}(t) }{2h_x},
 ~~j = 1,2,\cdots,N,   \label{eq:4.2}
 \end{flalign}
where $w_j(t)$ represents an approximate value of $w(x_j,t), ~x_j=jh_x$, and $w_0(t)=w_N(t),~w_{N+1}(t)=w_1(t)$.
Through Fourier analysis, it can be concluded that the eigenvalues of the coefficient matrix of equation  \eqref{eq:4.2} satisfy
\begin{align}
\lambda_k = \frac{2D}{h_x^2}(\cos(2k\pi h_x) - 1) - \frac{iA}{h_x}\sin(2k\pi h_x), \quad k = 1,\cdots,N. \label{eq:4.3}
\end{align}
The eigenvalues in \eqref{eq:4.3} are located on an ellipse in the left half-plane $\mathbb{C}^-$ (see Figure \ref{fig:Com2}).
From Figure \ref{fig:Com2}, it can be observed that
the slope of ellipse near the origin becomes progressively steeper as the Peclet number increases.

The stability domains of both the ARKC method \eqref{eq:2.14} and the PIROCK method \eqref{eq:2.19} can cover a suitable ellipse,
thus these two classes of methods are well-suited for solving the problem \eqref{eq:4.1} with appropriate Peclet number.
However, due to the evident limitations of their stability domains near the origin (see Figures \ref{fig:ARKCwdy} and \ref{fig:PIROCKwdy}),
both types of methods gradually become unsuitable for the problem \eqref{eq:4.1} as the Peclet number increases.

%

Because rectangles can cover their inscribed ellipses, our partitioned method \eqref{eq:3.1} can also solve the problem \eqref{eq:4.1} well.
Moreover, even if the Peclet number is relatively large, as long as appropriate $s$ and $m$ are selected, our method is still applicable.
This indicates that our method \eqref{eq:3.1} has a broader scope of application than the methods \eqref{eq:2.14} and \eqref{eq:2.19}.
Of course, in addition to the scope of application, computational efficiency is also an important focus that reflects the merits of a method.
Regarding the comparison of computational efficiency, we will leave it in the numerical experiment section.
In fact, the numerical results show that our method \eqref{eq:3.1} also performs very well in terms of computational efficiency (see Section \ref{sec.6} for details).
\begin{figure}[t]
	\centering
	\includegraphics[width=0.5\linewidth]{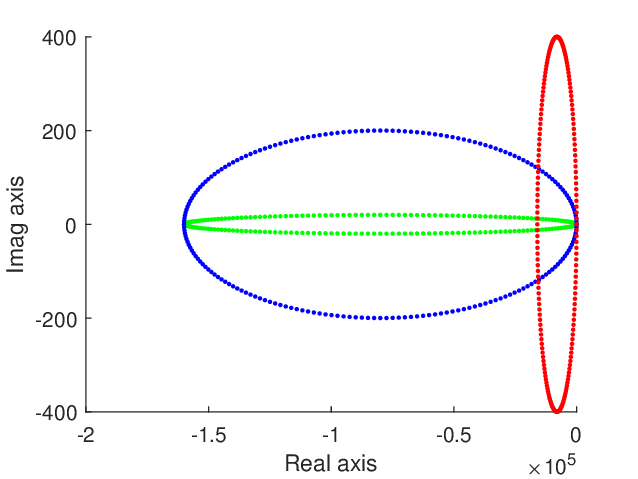}
	\caption{Distributions of eigenvalues in \eqref{eq:4.3}. $D=1, A = 0.1, N = 200$ (green dashed line), $D=1, A = 1, N = 200$ (blue dashed line), $D=0.1, A = 2, N = 200$ (red dashed line).}
	\label{fig:Com2}
\end{figure}

\subsection{ The damped wave problem}

Consider the 2-dimensional damped wave equation \cite{sommeijer2007}
\begin{flalign}
        Bw_t+w_{tt} &=A_1w_{xx}+A_2w_{yy}+D_1w_{txx}+D_2w_{tyy}+S(x,y),  ~~ x, y\in [0, 1],\label{eq:4.5}
\end{flalign}
where $S(x,y)$ is a source.
Let $\hat{w}=w_t$, then the equation \eqref{eq:4.5} can be rewritten as
\begin{flalign}
        w_t &=\hat{w},  \nonumber   \\
        \hat{w}_t &=-B\hat{w}+A_1w_{xx}+A_2w_{yy}+D_1\hat{w}_{xx}+D_2\hat{w}_{yy}+S. \label{eq:4.6}
\end{flalign}
Assuming that the coefficients $B, A_1, A_2, D_1, D_2$ are constants and the boundary conditions are periodic.
Similarly, by applying the second-order central difference formula to \eqref{eq:4.6}, a linear ODE system can be obtained, and the eigenvalues of the corresponding coefficient matrix satisfy
\begin{flalign}
	\lambda_{1,j_1, j_2} &= -\frac{1}{2}(\alpha_{D, j_1, j_2} + B) + \frac{1}{2}\sqrt{(\alpha_{D, j_1, j_2} + B)^2 - 4\alpha_{A, j_1, j_2}}, \nonumber \\
    \lambda_{2,j_1, j_2} &= -\frac{1}{2}(\alpha_{D, j_1, j_2} + B) - \frac{1}{2}\sqrt{(\alpha_{D, j_1, j_2} + B)^2 - 4\alpha_{A, j_1, j_2}},
 \nonumber \\
	\alpha_{A, j_1, j_2} &= \frac{4A_1}{h_x^2}\sin^2(j_1\pi h_x) + \frac{4A_2}{h_x^2}\sin^2(j_2\pi h_x),   \nonumber \\
	\alpha_{D, j_1, j_2} &= \frac{4D_1}{h_x^2}\sin^2(j_1\pi h_x) + \frac{4D_2}{h_x^2}\sin^2(j_2\pi h_x),
    \quad j_1, j_2\in\{1,2,\cdots, N\}. \label{eq:4.7}
\end{flalign}

Based on \eqref{eq:4.7}, taking $B=0, A_1=0.05, A_2=15, D_1=D_2=0.1$, then we can obtain Figure \ref{fig:Com3},
which shows that the damped wave problem \eqref{eq:4.5} demands high requirements on the width of the method's stability domain near the imaginary axis.
Because the stability domains of the methods \eqref{eq:2.14} and \eqref{eq:2.19} have a narrow width near the imaginary axis, these two types of methods are not suitable for solving the damped wave equation \eqref{eq:4.5}.
Fortunately, the stability domain of our methods approximates a rectangular domain with adjustable length and width, making it naturally suitable for solving the damped wave problem \eqref{eq:4.5}.

\begin{figure}[t]
	\centering
	\includegraphics[width=0.5\linewidth]{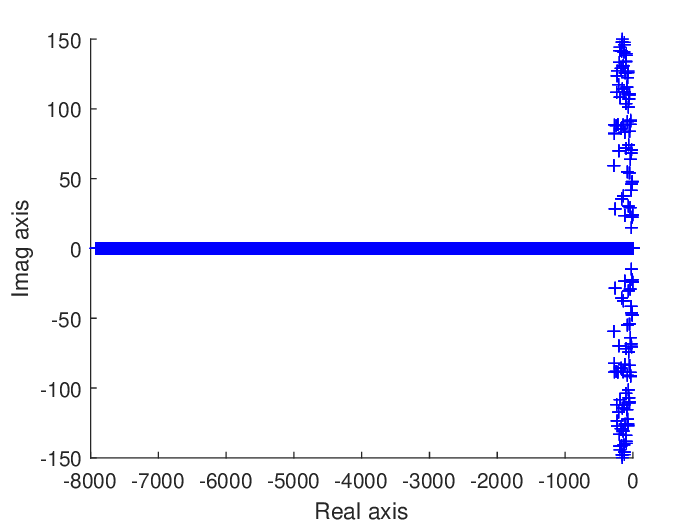}
	\caption{Distribution of eigenvalues in \eqref{eq:4.7} with $B=0, A_1=0.05, A_2=15, D_1=D_2=0.1$.}
	\label{fig:Com3}
\end{figure}

\section{Control of time step size and selection of $s$ and $m$}
\label{sec.5}

A good posterior error estimation is crucial for adaptive step size selection.
Therefore, we first introduce how to obtain an appropriate posterior error.
Inspired by reference \cite{zbinden2011}, we first give two separate error estimators for $f_D$ and $f_A$, and then determine the final posteriori error.

For the error estimator about $f_D$, we will introduce two strategies.
Inspired by reference \cite{sommeijer1998},
the first error estimation strategy can be rewritten as
\begin{flalign}
err_D = \frac{1}{15}\Big{(}12(K_0-K_s)+6h \big{(}F_D(K_0)+F_D(K_s)\big{)}\Big{)}, \label{eq:5.1}
\end{flalign}
where $K_0$ and $K_s$ are defined by method \eqref{eq:3.1}.
Using the above estimation, an additional evaluation $F_D(K_s)$ is required.
According to \cite{sommeijer1998}, the estimation $err_D$ is approximately equal to the local error of
the second-order RKC methods, thus we have $err_D=\mathcal{O}(h^3)$.
This error estimation tends to be saturated, making it prone to result in the final accuracy of adaptive methods falling short of the given precision requirements (see \cite{zbinden2011,kleefeld2013} and Section \ref{sec.6}).

The second error estimation strategy is a new strategy proposed by us.
This strategy use the idea of embedded methods.
We first construct the following embedded method
\begin{flalign}
	\tilde{K}_s= (1-\frac{1}{b_{s_1} T_{s_1}'(\omega_0)\omega_1})K_0
   +\frac{1}{b_{s_1} T_{s_1}'(\omega_0)\omega_1} K_{s_1}, \label{eq:5.2}
\end{flalign}
where $s_1 = \lfloor \frac{4}{5}s \rfloor$ ($\lfloor \cdot \rfloor$ is a down rounding function), $K_0$ and $K_{s_1}$  are defined by method \eqref{eq:3.1}.
Next, we give the second error estimation strategy, which can be defined as
\begin{flalign}
	\widetilde{err}_D= K_s-\tilde{K}_s. \label{eq:5.3}
\end{flalign}
It is easy to verify that the embedded method \eqref{eq:5.2} is first-order consistent when $f_A\equiv 0$.
This means $\widetilde{err}_D=\mathcal{O}(h^2)$.

For the error estimator about $f_A$, we also use the idea of embedded methods.
We construct the following embedded method
\begin{flalign}
	K^{*}_0 &= K_s,\nonumber \\
	K^{*}_{i} &= K^{*}_{i-1}-\frac{1}{m}h F_A(K_{s+3i-3})+\frac{3}{2m}h F_A(K_{s+3i-2}), \quad  1 \leq \ i \leq m, \nonumber \\
	\tilde{y}_{n+1} &=K^{*}_{m+1},  \label{eq:5.4}
\end{flalign}
where $K_s, K_{s+3i-3}, K_{s+3i-2}$ are defined by method \eqref{eq:3.1}.
It is not difficult to verify that the embedded method \eqref{eq:5.4} is also second-order consistent for the ODE \eqref{eq:1.1}.
Therefore, if take
\begin{flalign}
	err_A= y_{n+1}-\tilde{y}_{n+1}, \label{eq:5.5}
\end{flalign}
then we have $err_A=\mathcal{O}(h^3)$.

Based on estimations $err_D, \widetilde{err}_D$ and $err_A$, we can define the posterior error of method \eqref{eq:3.1} as
\begin{flalign}
	err= \max (\|err_D\|,~\|err_A\|), \label{eq:5.6}
\end{flalign}
or
\begin{flalign}
	err= \max (\|\widetilde{err}_D\|,~\|err_A\|^{2/3}). \label{eq:5.7}
\end{flalign}
For the time step size control, we use the strategy 
\begin{flalign}
	h_{new}= fac\cdot h\cdot (tol/err)^{1/p}, \label{eq:5.8}
\end{flalign}
where $fac=0.8$ is a safe factor, $tol$ is the specified tolerance,
$p=3$ if $err$ is defined by \eqref{eq:5.6} and $p=2$ if $err$ is defined by \eqref{eq:5.7}.

After the time step size $h$ is determined, we can proceed to select $s$ and $m$.
Specifically, based on Theorem \ref{thm1}, $s$ and $m$ can be determined by
\begin{flalign}
	s = \left\lceil \sqrt{\frac{h}{0.65} \rho_D+1} ~\right\rceil ,\qquad \qquad m = \left\lceil {\frac{h}{2.15} \rho_A}  \right\rceil,  \label{eq:5.9}
\end{flalign}
where $\lceil \cdot \rceil$ is an upward rounding function, $\rho_D$ and $\rho_A$
represent the spectral radii of the Jacobian matrices of functions $f_D$ and $f_A$, respectively.

\section{Numerical results}
\label{sec.6}

In this section, we will perform a numerical comparison of the methods introduced earlier through several specific examples. For ease of description, our partitioned method \eqref{eq:3.1} will be abbreviated as NPRKC method. In particular, when posterior error estimation \eqref{eq:5.6} is used, the method \eqref{eq:3.1} is called NPRKC1 method, whereas when posterior error estimation \eqref{eq:5.7} is used, the method \eqref{eq:3.1} is called NPRKC2 method.

The computational complexity of several methods is given in Table \ref{tab:1},
where $NfD$ denotes the number of evaluations of $f_D$ in each step and
$NfA$ denotes the number of evaluations of $f_A$ in each step.
\begin{table}
	\centering
	\caption{~Computational complexity }
	\label{tab:1}
	\begin{tabular}{lcc}
		\hline
		methods & ~~$NfD$~~ & ~~$NfA$~~\\
		\hline
        RKC & s & s \\
		PRKC & s & 4 \\
        ARKC & s+2 & 3 \\
		PIROCK & s+2+$\ell$ & 3 \\
		NPRKC1 & s & 4m \\
		NPRKC2 & s & 4m \\
		\hline
	\end{tabular}
\end{table}

For all the examples presented in this section, we have employed the second-order central difference formula for spatial semi-discretization.
Unless otherwise specified, we denote the part arising from spatial semi-discretization of the diffusion terms as $f_D$, and the remaining part as $f_A$.
For problems lacking exact solutions, we adopt the numerical solution computed by the DOPRI5 method \cite{hairer1993} with small step size $h=2^{-18}$ as the reference solution.

\begin{example}
\label{ex1}
For the first example, we consider the 2-dimensional damped wave equation \eqref{eq:4.5}.
Assuming that the equation is defined on the unit square and satisfies zero-flux boundary conditions.
Let $ t\in[0, ~0.75], ~w(x,y,0)=w_t(x,y,0)=0, ~B=0,  ~A_1=0.05, ~A_2=15$, and
\begin{flalign}
	D_1=D_2=D(x,y) &= 0.1 \mathrm{e}^{-100((x - 0.25)^2 + (y - 0.25)^2)}, \nonumber \\
	S(x,y) &= 100 \mathrm{e}^{-500((x - 0.75)^2 + (y - 1)^2)} + 100 \mathrm{e}^{-500((x - 0.25)^2 + (y - 1)^2)}. \label{eq:6.1}
\end{flalign}
Though the coefficients $D_1$ and $D_2$ are no longer constant in this case, the pattern of eigenvalue distribution is still similar to that shown in Figure \ref{fig:Com3}.
In this example, we denote the part obtained from the semi-discretization of $D_1\hat{w}_{xx}+D_2\hat{w}_{yy}$ in \eqref{eq:4.6} as $f_D$, while designating the remaining part as $f_A$.
Then, it can be seen from \cite{sommeijer2007} that
\begin{flalign}
	\rho_A\leq 2N\sqrt{A_1+A_2},~~~~\rho_D\leq 8N^2Q+B, ~~~~Q=\max\limits_{0\leq x, ~y\leq 1}D(x,y). \label{eq:6.2}
\end{flalign}
\end{example}

Since the intersection of the stability domains of the methods RKC and ARKC with the imaginary axis is only the origin, these two types of methods cannot be applied to solve the current example.
In this example, we mainly compare the methods PRKC, PIROCK and NPRKC.
For a more intuitive comparison of the stability of several methods on the imaginary axis and its vicinity, we first consider the case of fixed step size.
Let $N=100$, then the corresponding numerical results are shown in Figures \ref{fig:wavezj}-\ref{fig:NPRKCsloveWave}.
As shown in Figures \ref{fig:PRKCsloveWave} and \ref{fig:PIROCKsloveWave}, the methods PRKC and PIROCK suffer from strict step size restriction.
Fortunately, Figure \ref{fig:NPRKCsloveWave} shows that our NPRKC method can effectively avoid the step size restriction problem.
In addition, by comparing the middle subfigures of Figures \ref{fig:PRKCsloveWave}-\ref{fig:NPRKCsloveWave}, it can be observed that even when $m=1$,
the performance of our NPRKC method is superior to that of both the PRKC and PIROCK methods.

Next, we will consider the comparison of variable step sizes.
Let $Err$ represent numerical error, $Nac(Nre)$ denote the number of acceptance (rejection) steps, and $NfDe,~NfAe$ denote the counts of computations for $f_D,~f_A$, respectively.
Take $N=100$ and $tol=10^{-1}, 10^{-2}, \cdots, 10^{-5}$ in turn, then we can obtain Table \ref{tab:2}.

From Table \ref{tab:2}, it can be observed that when $tol\leq 10^{-4}$, the methods PRKC and PIROCK have significantly more computational steps than our NPRKC methods.
This once again confirms the step size restriction issue of the methods PRKC and PIROCK.
Comparing the methods NPRKC1 and NPRKC2, it can be concluded that the numerical results of the NPRKC2 method used posterior error estimation \eqref{eq:5.7} are more reliable as its errors are consistently smaller than the specified tolerance $tol$.
Comparing the methods PRKC, PIROCK and NPRKC2, it is not difficult to find that our NPRKC2 method has a lower computational cost while meeting the accuracy requirements.

\begin{figure}[H]
	\centering
	\includegraphics[width=0.35\linewidth]{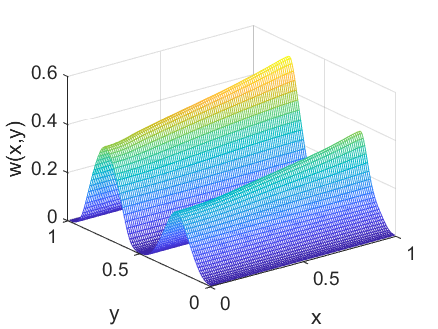}
	\caption{Reference solution of the Example \ref{ex1} at $t=0.75$.}
	\label{fig:wavezj}
\end{figure}

\begin{figure}[H]
	\centering
	\begin{minipage}[b]{0.30\textwidth}
		\includegraphics[width=\textwidth]{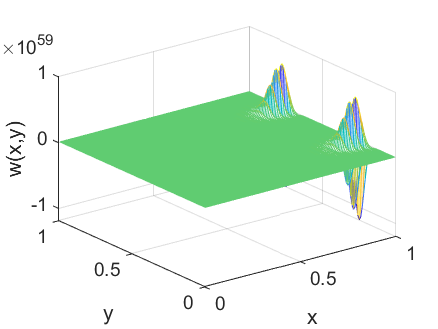} 	
	\end{minipage}
	\hfill 
	\begin{minipage}[b]{0.30\textwidth}
		\includegraphics[width=\textwidth]{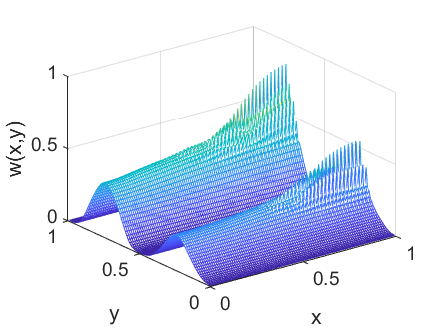}  
	\end{minipage}
	\hfill 
	\begin{minipage}[b]{0.30\textwidth}
		\includegraphics[width=\textwidth]{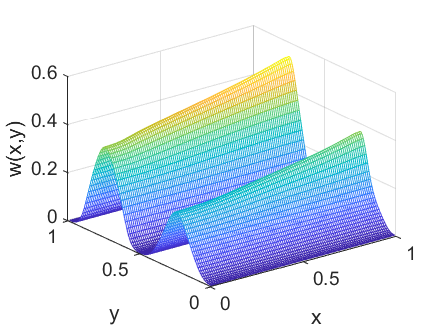}  
	\end{minipage}
	\caption{ Numerical solutions of the PRKC method for the Example \ref{ex1}. $h=1/30, s=22$ (left), $h=1/400, s=7$ (middle) and $h=1/450, s=7$ (right).}
	\label{fig:PRKCsloveWave}
\end{figure}

\begin{figure}[H]
	\centering
	\begin{minipage}[b]{0.30\textwidth}
		\includegraphics[width=\textwidth]{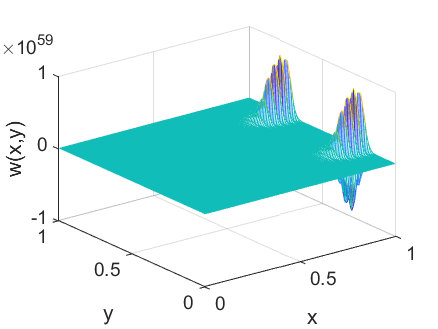} 	
	\end{minipage}
	\hfill 
	\begin{minipage}[b]{0.30\textwidth}
		\includegraphics[width=\textwidth]{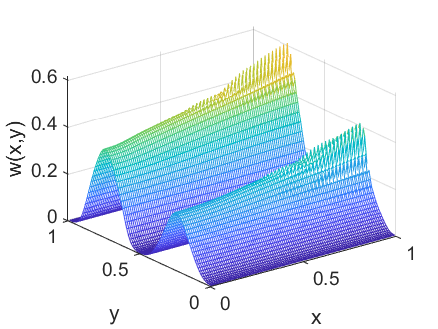}  
	\end{minipage}
	\hfill 
	\begin{minipage}[b]{0.30\textwidth}
		\includegraphics[width=\textwidth]{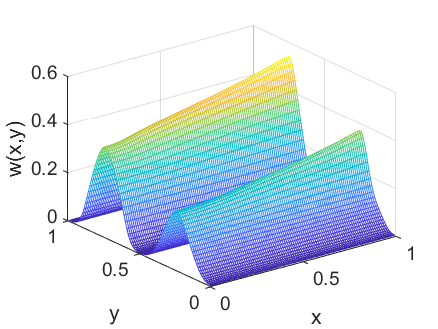}  
	\end{minipage}
	\caption{ Numerical solutions of the PIROCK method for the Example \ref{ex1}. $h=1/30, s=27$ (left), $h=1/400, s=8$ (middle) and $h=1/450, s=8$ (right).}
	\label{fig:PIROCKsloveWave}
\end{figure}

\begin{figure}[H]
	\centering
	\begin{minipage}[b]{0.30\textwidth}
		\includegraphics[width=\textwidth]{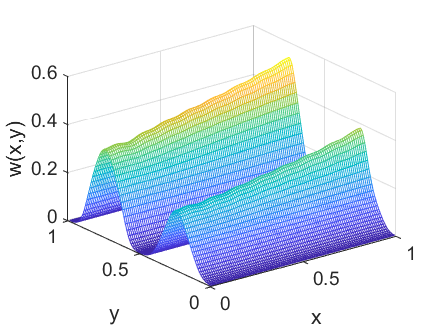} 	
	\end{minipage}
	\hfill 
	\begin{minipage}[b]{0.30\textwidth}
		\includegraphics[width=\textwidth]{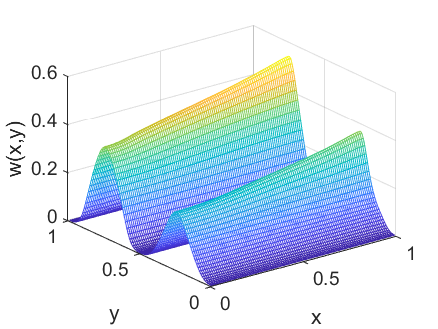}  
	\end{minipage}
	\hfill 
	\begin{minipage}[b]{0.30\textwidth}
		\includegraphics[width=\textwidth]{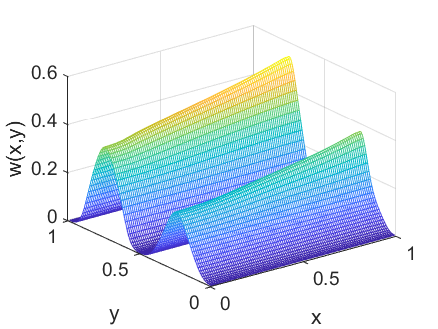}  
	\end{minipage}
	\caption{ Numerical solutions of our NPRKC method for the Example \ref{ex1}. $h=1/30, s=22, m=13$ (left), $h=1/400, s=7, m=1$ (middle) and $h=1/450, s=7, m=1$ (right).}
	\label{fig:NPRKCsloveWave}
\end{figure}

\begin{table}[t]
	\small
	\setlength{\tabcolsep}{3pt}
	\centering
	\caption{The results of the Example \ref{ex1}.}
	\label{tab:2}
	\begin{tabular}{lccccccc}
			\hline
			tol & method & $Err$ & $Nac(Nre)$  & $NfDe$ &$NfAe$ &$ NfDe+NfAe $\\
			\hline
			$10^{-1}$
			& PRKC & 2.8828e--05 & 343(0) &2398 &1372 & 3770\\
			&PIROCK & 2.9931e+07  & 340(7) &3792 &1191 & 4983\\
			& NPRKC1 &2.0393e--03  & 14(32) &439 &1760 & 2199\\
			& NPRKC2 & 6.3715e--02  & 8(2) &358 &1868 & 2226\\
			$10^{-2}$
			& PRKC & 2.8828e--05  & 343(0) &2398 &1372& 3770\\
			&PIROCK & 8.9671e--04  & 309(2) &3265 &933& 4198\\
			& NPRKC1 & 2.0558e--03 & 48(5) &758 &1516 & 2284\\
			& NPRKC2 & 2.6174e--03  & 34(8) &756 &1612 & 2368\\
			$10^{-3}$
			& PRKC & 2.8828e--05  & 343(0) &2398 &1372& 3770\\
			&PIROCK & 5.6145e--04 & 318(4) &3400 &966 & 4366\\
			& NPRKC1 & 5.5325e--04  & 71(1) &839&1276& 2115\\
			& NPRKC2 & 5.1916e--04  & 69(1) &816 &1236 & 2052\\
			$10^{-4}$
			& PRKC & 2.1884e--05  & 389(0) &3412 &1556 & 3968\\
			&PIROCK & 1.7443e--05  & 321(3) &3371 &972 & 4343\\
			& NPRKC1 & 1.2598e--04  & 158(0) &1275 &1830 & 2655\\
			& NPRKC2 & 5.3196e--05  & 210(1) &1464 &1404 & 2868\\
			$10^{-5}$
			& PRKC & 6.5451e--06  & 797(1) &3980 &3192 & 7172\\
			&PIROCK & 7.3758e--06   & 527(1) &4725 &1584 & 6309\\
			& NPRKC1 & 1.6935e--05   & 518(3) &2608 &2096 & 4704\\
			& NPRKC2 & 7.7131e--06   & 697(6) &2924 &2824 & 5748\\
			\hline
		\end{tabular}
	
\end{table}

\begin{example}
\label{ex2}
In this example, we consider the advection-diffusion equation \eqref{eq:4.1}.
We readily obtain that $\rho_D=4|D|N^2$ and $\rho_A=|A|N$.
Let $T=\frac{1}{10},~N=200$ and $w_0(x)=sin(2\pi x)$.
The corresponding numerical results are shown in Tables \ref{tab:3}-\ref{tab:5}.

As shown in Table \ref{tab:3}, for the case of small Peclet number,
all partitioned methods PRKC, ARKC, PIROCK and NPRKC perform well and outperform the standard RKC method.
However, as the Peclet number increases, the PRKC method begins to encounter the problem of step size restriction
(see the results of Tables \ref{tab:4} and \ref{tab:5} with $tol=10^{-2}$).
In addition, when the Peclet number reaches a certain level, the methods ARKC and PIROCK also become unsuitable.
Their errors significantly exceeded the specified tolerance, particularly when $tol=10^{-5}$ (see Table \ref{tab:5}).
Fortunately, our methods NPRKC1 and NPRKC2 perform well in all situations.

\begin{table}[H]
	\small
	\setlength{\tabcolsep}{3pt}
	\centering
	\caption{The results of the Example \ref{ex2} with $A = 0.1, D = 1$.}
	\label{tab:3}
	\begin{tabular}{lccccccc}
		\hline
		tol & method & $Err$  & $Nac(Nre)$ & $NfDe$ &$NfAe$ &$ NfDe+NfAe $\\
		\hline
		$10^{-2}$ & RKC & 2.1589e--03 & 8(0) & 426 & 426 &858 \\
		& PRKC & 2.1550e--03  & 8(0) & 434 & 32& 466 \\
		 &PIROCK & 1.0015e--03 & 8(0)& 545 &24  &569  \\
		 & ARKC & 1.9582e--03  & 9(0) & 489 & 18 & 509 \\
		 & NPRKC1 & 2.1550e--03  & 8(0) & 434 & 32 & 466\\
		  & NPRKC2 & 1.2977e--03  & 10(0) & 491 & 40 & 531\\
	$10^{-5}$ & RKC & 2.6849e--05  & 54(0) & 1166 & 1166 & 2332 \\
	& PRKC & 2.6832e--05  & 54(0) & 1221 & 216 & 1437 \\
	&PIROCK &2.4597e--06  & 160(1) & 2972 &483 &3455\\
	& ARKC & 2.1437e--05  & 60(0) & 1420 & 120 & 1440 \\
	& NPRKC1 & 2.6832e--05  & 54(0) & 1221 & 216 & 1437 \\
	& NPRKC2 & 2.1540e--06  & 229(1) & 2655 & 920 & 3575 \\
		\hline
	\end{tabular}
	
\end{table}
\vspace{-0.5cm}
\begin{table}[H]
	\small
	\setlength{\tabcolsep}{3pt}
	\centering
	\caption{The results of the Example \ref{ex2} with $A = 5, D= 1$.}
	\label{tab:4}
	\begin{tabular}{lccccccc}
		\hline
		tol & method & $Err$  & $Nac(Nre)$ & $NfDe$ &$NfAe$ &$ NfDe+NfAe $\\
		\hline
		$10^{-2}$ & RKC & 2.5539e--03  & 10(0) & 481 & 481 &962 \\
		& PRKC & 2.3288e--05  &60(0) & 1304 & 240 & 1544\\
		&PIROCK &5.3304e--03  & 8(0)& 546 &24 &570\\
		& ARKC & 1.2535e--03  & 8(1) & 805 & 18 & 823 \\
		& NPRKCA1 &2.1688e--03  & 8(0) & 426 & 192 & 622\\
		& NPRKCA2& 1.3058e--03  & 10(0) & 491 & 200 & 691\\
		$10^{-5}$ & RKC & 3.3985e--05  & 70(0) & 1329 & 1329 & 2658 \\
		& PRKC & 1.8828e--05  & 69(0) & 1402 & 276& 1678 \\
		&PIROCK &1.1730e--05   & 160(1) & 2972 &483 & 3455\\
		& ARKC & 1.8720e--05  & 51(33) & 2967 & 168 & 3135 \\
		& NPRKCA1 & 2.6743e--05  & 54(0) & 1167 & 272 & 1439 \\
		& NPRKCA2 & 2.1452e--06  & 229(1) & 2655 & 920 & 3575 \\
		\hline
	\end{tabular}
\end{table}
\vspace{-0.5cm}
\begin{table}[H]
	\small
	\setlength{\tabcolsep}{3pt}
	\centering
	\caption{The results of the Example \ref{ex2} with $A = 5, D = 0.2$.}
	\label{tab:5}
	\begin{tabular}{lccccccc}
		\hline
		tol & method & $Err$  & $Nac(Nre)$ &  $NfDe$ &$NfAe$ &$ NfDe+NfAe $\\
		\hline
		$10^{-2}$ & RKC & 1.4227e--02  & 9(0) & 207 & 207 &414 \\
		& PRKC & 18344e--04  &60(0) & 653 & 240& 893 \\
		&PIROCK &4.6558e--02    &7(0)  &229&21 &250 \\
		& ARKC & 2.2955e--02  & 10(0) & 335& 20 & 355 \\
		& NPRKC1 & 3.0295e--03  & 5(0) & 142 & 196 & 338\\
		& NPRKC2 & 1.1741e--03  & 5(0) & 148 & 192 & 340\\
		$10^{-5}$ & RKC & 1.7369e--04  & 67(0) & 604& 604 & 1208 \\
		& PRKC &7.5495e--05  & 92(0) & 822 & 368 & 1191 \\
		&PIROCK &2.1616e--04 & 56(0)& 850 &168 &1018\\
		& ARKC & 2.5561e--03  & 79(0) & 998 & 150 & 1156 \\
		& NPRKC1 & 3.7919e--06  & 40(0) & 503 & 212 & 715 \\
		& NPRKC2 & 3.8247e--07  & 76(0) & 717 & 304 & 1021\\
		\hline
	\end{tabular}
\end{table}
\end{example}

\begin{example}
\label{ex3}
Consider the 2-dimensional Burrelator problem with periodic boundary \cite{abdulle2013}
\begin{flalign}	
w_t &= D( w_{xx}+w_{yy})+A ( -0.5w_{x}+w_{y})+ w^{2}\hat{w}-2w+1.3,\nonumber \\
\hat{w}_t &= D ( \hat{w}_{xx}+\hat{w}_{yy})+A (0.4\hat{w}_{x}+0.7\hat{w}_{y}) + w - w^{2}\hat{w},\nonumber \\	
w(x,y,0) &= 22y(1-y)^{3/2},~~~~x,y\in[0, ~1],~~~~t\in[0, ~1],\nonumber \\
\hat{w}(x,y,0) &= 22x(1-x)^{3/2}.\label{eq:6.3}
\end{flalign}
Let $N=200$, and consider two distinct sets of parameters $D=0.04,~A=0.02$ and $D=0.04,~A=2$.
The corresponding numerical results are shown in Figures \ref{fig:Bur1} and \ref{fig:Bur2}.

From Figure \ref{fig:Bur1}, we can analyze whether the actual accuracy of each numerical method can achieve the specified tolerance.
It is not difficult to observe that our methods NPRKC1 and NPRKC2 demonstrate good performance, with the method NPRKC2 in particular meeting the accuracy requirements across all cases.
However, as can be observed from the right subfigure of Figure \ref{fig:Bur1}, the methods PRKC, ARKC and PIROCK all exhibit certain issues for the second set of parameters.
For the PRKC method, the step size restriction from stability causes its actual accuracy to remain almost unchanged with varying $tol$.
For the methods ARKC and PIROCK, their actual accuracy is far from reaching the specified tolerance.
As can be seen from Figure \ref{fig:Bur2}, our methods NPRKC1 and NPRKC2 also demonstrate advantages from the perspective of computational efficiency,
and the advantage become more pronounced in case of the second set of parameters.

\begin{figure}[H]
	\centering
	\begin{minipage}[b]{0.48\textwidth}
		\includegraphics[width=\textwidth]{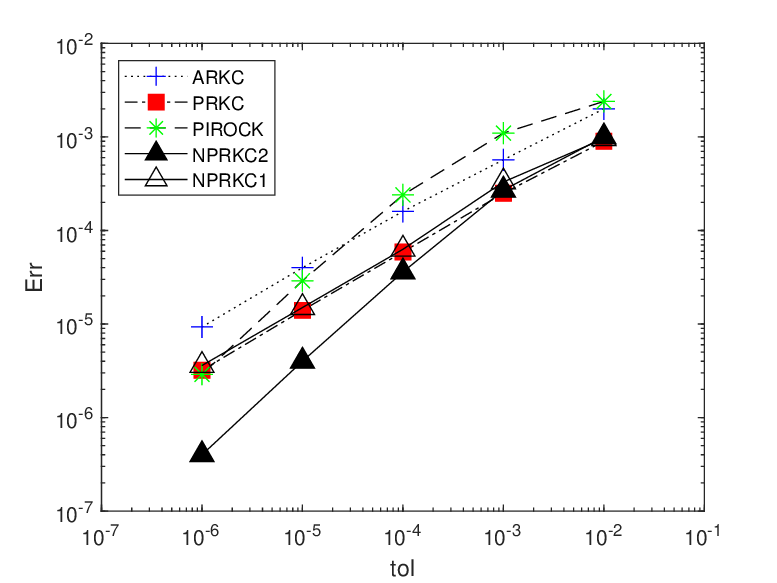} 	
	\end{minipage}
	\hfill 
	\begin{minipage}[b]{0.48\textwidth}
		\includegraphics[width=\textwidth]{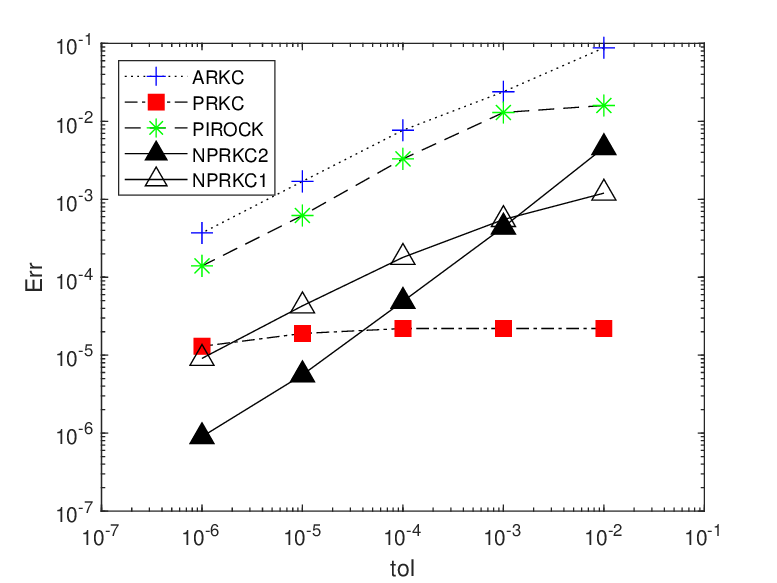}   
	\end{minipage}
	\caption{The results of Example \ref{ex3}: $tol$ $vs.$ $Err$. $D = 0.04$, $A = 0.02$ (left) and $D = 0.04$, $A = 2$ (right). }
	\label{fig:Bur1}
\end{figure}
\begin{figure}[H]
	\centering
	\begin{minipage}[b]{0.48\textwidth}
		\includegraphics[width=\textwidth]{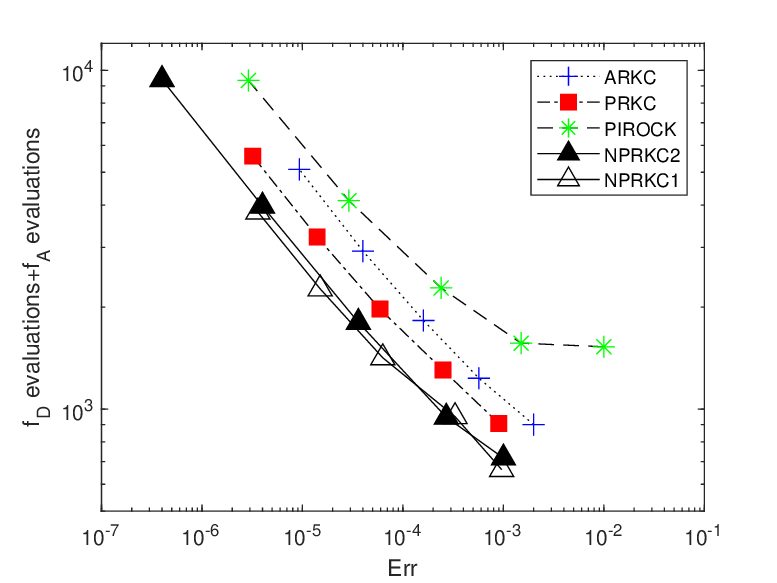}	
	\end{minipage}
	\hfill 
	\begin{minipage}[b]{0.48\textwidth}
		\includegraphics[width=\textwidth]{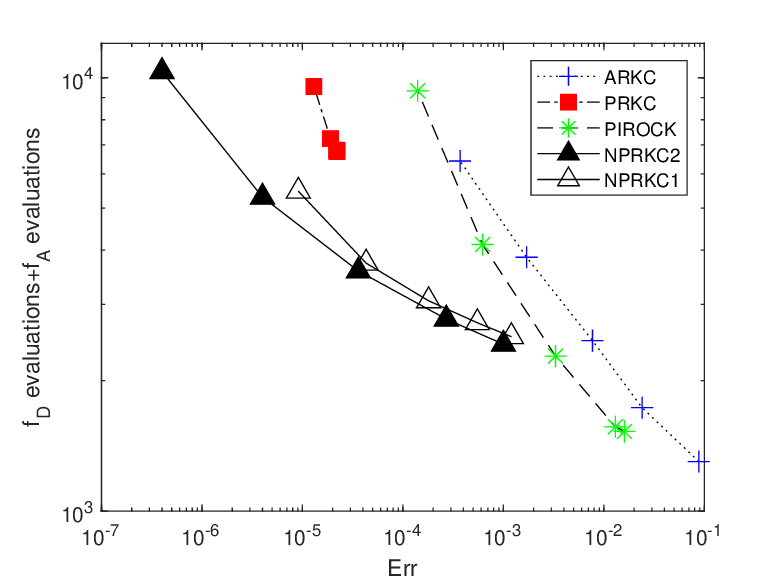}   
	\end{minipage}
	\caption{The results of Example \ref{ex3}: $Err$ $vs.$ $cost$. $D = 0.04$, $A = 0.02$ (left) and $D = 0.04$, $A = 2$ (right). }
	\label{fig:Bur2}
\end{figure}

\end{example}

\begin{example}
\label{ex4}
 Consider the 1-dimensional Burgers equation with periodic boundary \cite{li2023implicit,almuslimani2023}
\begin{flalign}
	w_t &= D w_{xx}+A ww_x , \quad (x,t)\in[0,~1]\times[0,~0.5], \nonumber \\
    w(x, 0) &=1+cos(2\pi x). \label{eq:6.4}
\end{flalign}
Let $D=0.5, ~A=10$ and $N=100$.
Under this set of parameters, the comparisons will be conducted among the methods ARKC, PIROCK and NPRKC.
The corresponding numerical results are shown in Figure \ref{fig:burger1D}.

\begin{figure}[H]
	\centering
	\begin{minipage}[b]{0.48\textwidth}
		\includegraphics[width=\textwidth]{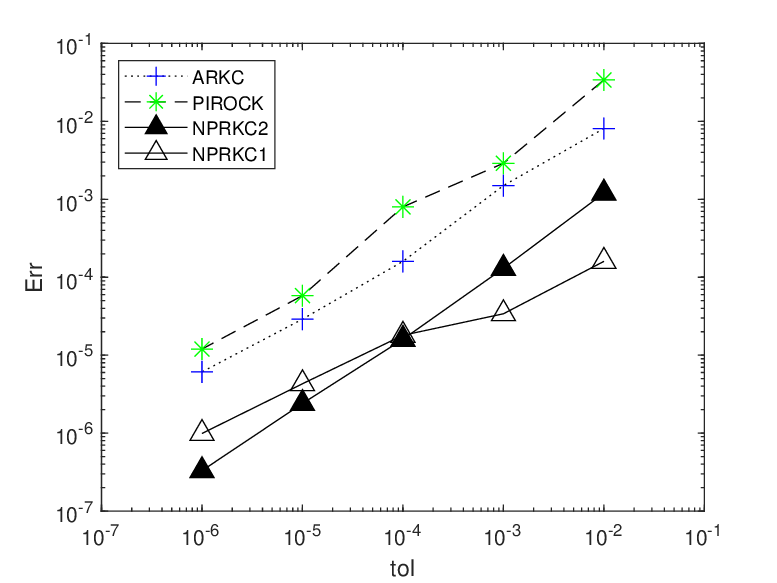} 	
	\end{minipage}
	\hfill 
	\begin{minipage}[b]{0.48\textwidth}
		\includegraphics[width=\textwidth]{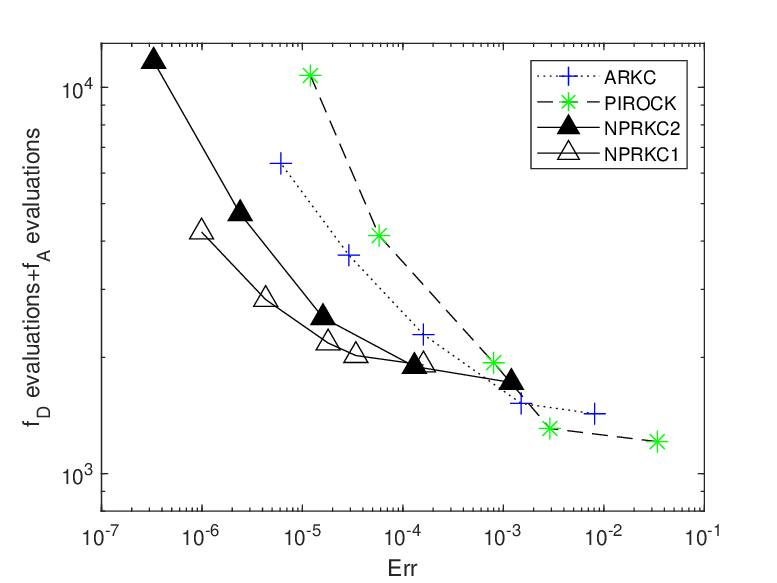}  
	\end{minipage}
	\caption{The results of Example \ref{ex4}: $tol$ $vs.$ $Err$ (left) and $Err$ $vs.$ $cost$ (right)}
	\label{fig:burger1D}
\end{figure}

\end{example}

\begin{example}
\label{ex5}
 Consider the 2-dimensional Burgers equation with periodic boundary \cite{zhao2011new}
\begin{flalign}
     w_t &= D( w_{xx}+w_{yy}) +A (ww_{x}+\hat{w}w_{y}),\nonumber \\
	 \hat{w}_t &= D( \hat{w}_{xx}+\hat{w}_{yy}) +A (w\hat{w}_{x}+\hat{w}\hat{w}_{y}),  \quad (x,y,t)\in[0, ~1]\times[0,~1]\times[0, ~0.5],\nonumber \\
	 w(x,y,0)&=1+\cos(2\pi x)\cos(2\pi y),\nonumber \\
	 \hat{w}(x,y,0)& = 1+\sin(2\pi x)\sin(2\pi y). \label{eq:6.5}
\end{flalign}
Let $D=0.2, ~A=4$ and $N=100$. The corresponding numerical results are shown in Figures \ref{fig:burgers2D}-\ref{fig:Burgerzj2}.

As shown in Figures \ref{fig:burger1D} and \ref{fig:burgers2D},
the methods ARKC and PIROCK still exhibit the issue where the actual numerical errors exceed the specified tolerance,
and this problem becomes more pronounced in the 2-dimensional Burgers equation.
For our methods NPRKC1 and NPRKC2, they perform well in terms of both computational accuracy and efficiency.

\begin{figure}[H]
	\centering
	\begin{minipage}[b]{0.48\textwidth}
		\includegraphics[width=\textwidth]{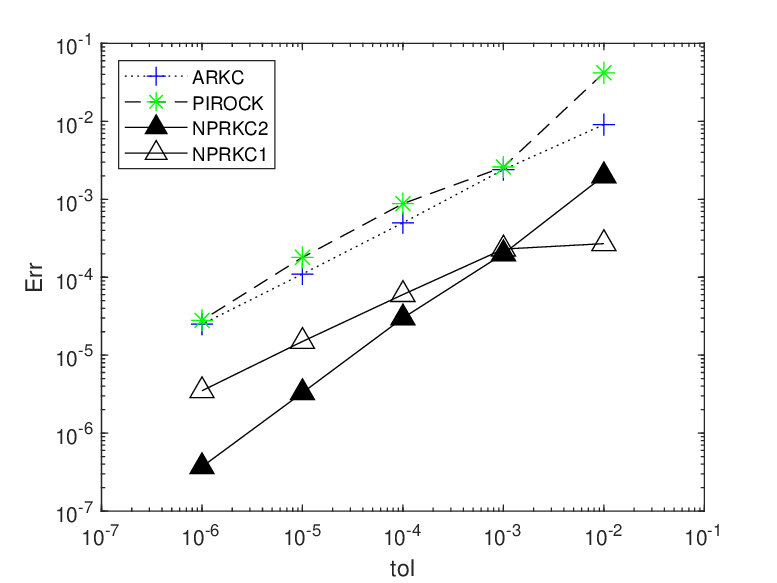} 	
	\end{minipage}
	\hfill 
	\begin{minipage}[b]{0.48\textwidth}
		\includegraphics[width=\textwidth]{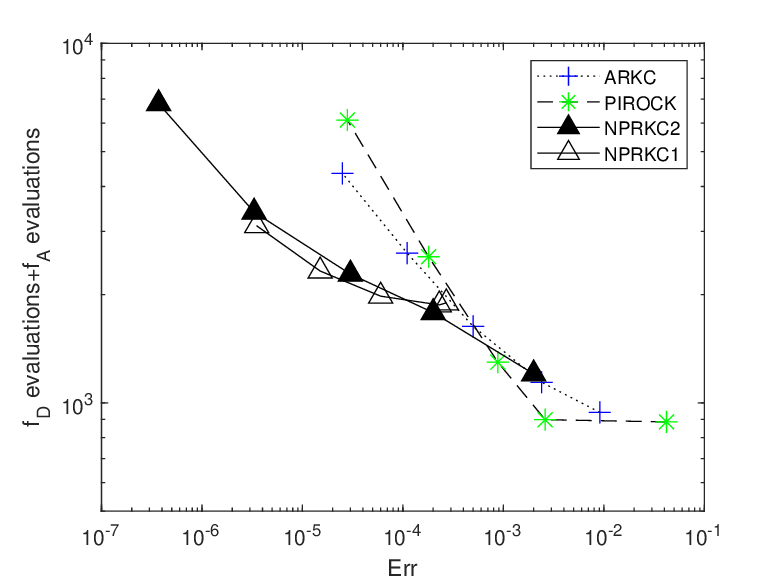}   
	\end{minipage}
	\caption{The results of Example \ref{ex5}: $tol$ $vs.$ $Err$ (left) and $Err$ $vs.$ $cost$ (right) }
	\label{fig:burgers2D}
\end{figure}

\begin{figure}[H]
	\centering
	\begin{minipage}[b]{0.48\textwidth}
		\includegraphics[width=\textwidth]{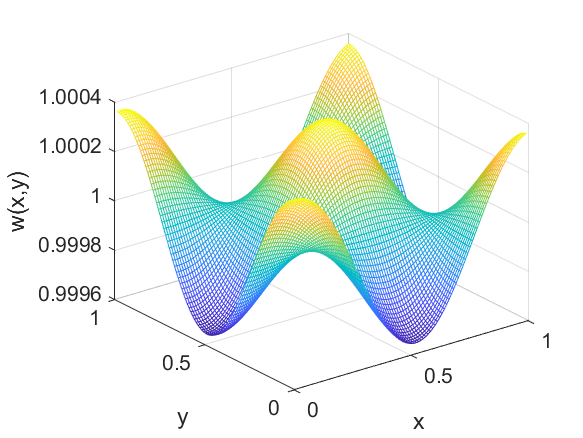}	
	\end{minipage}
	\hfill 
	\begin{minipage}[b]{0.48\textwidth}
		\includegraphics[width=\textwidth]{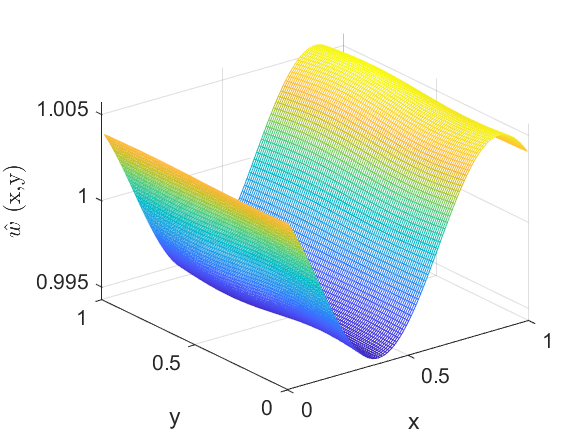}   
	\end{minipage}
	\caption{Reference solutions $w$ (left) and $\hat{w}$ (right) of Example \ref{ex5} at $t = 0.5$ }
	\label{fig:Burgerzj1}
\end{figure}

\begin{figure}[H]
	\centering
	\begin{minipage}[b]{0.48\textwidth}
		\includegraphics[width=\textwidth]{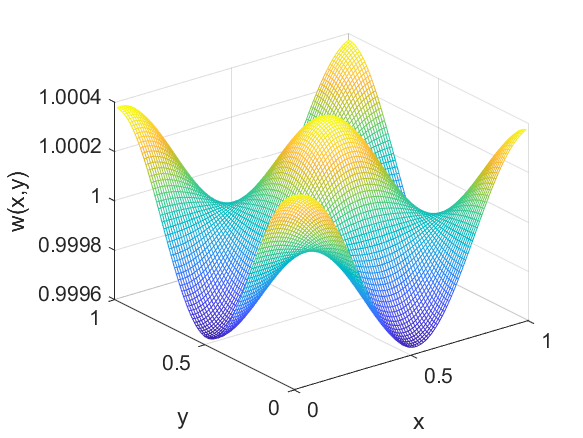}	
	\end{minipage}
	\hfill 
	\begin{minipage}[b]{0.48\textwidth}
		\includegraphics[width=\textwidth]{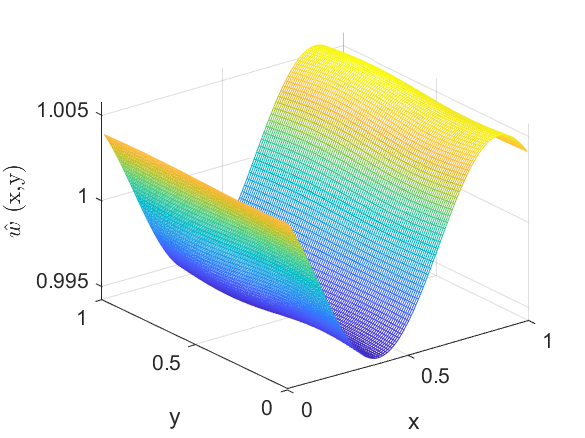}   
	\end{minipage}
	\caption{The numerical solutions $w$ (left) and $\hat{w}$ (right) of the NPRKC1 method with $tol=10^{-5}$. }
	\label{fig:Burgerzj2}
\end{figure}

\end{example}

\section{Conclusion}
\label{sec.7}

In this paper, a new class of second-order partitioned RKC methods are proposed for the ODEs containing moderately stiff and non-stiff terms.
We treat the moderately stiff term with an $s$-stage RKC method and treat the non-stiff term with a $4m$-stage explicit RK method,
where both the parameters $s$ and $m$ can be flexibly adjusted according to the needs of the problems.
We analyzed the convergence and stability of the constructed methods, and presented the variable step-size control strategies as well as the selection strategy for the parameters $s$ and $m$.
Compared to the existing partitioned explicit stabilized methods PRKC, ARKC and PIROCK,
our new methods offer advantages across multiple aspects, including applicability, computational accuracy and efficiency.

\section*{Acknowledgments}
This research is supported by the National Science Foundation of China (No. 12101525) and
the Natural Science Foundation of Hunan Province of China (No. 2023JJ40615).

\bibliographystyle{tfs}
\bibliography{NPRKCreferences}

\end{document}